\documentclass{amsart}
  \usepackage[top=1in, bottom=1in, left=1.375in, right=1.375in]{geometry}
 \usepackage{amsmath,amssymb}
 \usepackage{algpseudocode}
 \usepackage{algorithm}
 \usepackage{algorithmicx}
 \usepackage{caption}
 \usepackage{subcaption}
 \usepackage{amsthm}
 \numberwithin{equation}{section}
 \usepackage{amsfonts}
 \usepackage{graphicx}
 \usepackage{hyperref}

\usepackage{amsmath}
\usepackage{amssymb}
\usepackage{amsthm}
\usepackage{amsmath}
\usepackage{setspace}
\usepackage{xcolor}
\usepackage{fancyhdr}
\usepackage{amssymb}
\usepackage{amsthm}
\usepackage{listings}
    \usepackage{cite}
    \usepackage{tabularx}
    \usepackage{graphicx}
    \usepackage{epstopdf}    
    \usepackage{epsfig}
    \usepackage{float}
    \usepackage{ listings}
    \usepackage{appendix}
 \theoremstyle{plain}
  \newtheorem{thm}{Theorem}
 \newtheorem{prop}{Proposition}[section]
 \newtheorem{lem}[prop]{Lemma}
 
 \theoremstyle{definition}
 \newtheorem{definition}[prop]{Definition}
 
 \theoremstyle{remark}
 \newtheorem{remark}[prop]{Remark}

 \let\pa=\partial
 \let\al=\alpha
 \let\b=\beta
 \let\d=\delta
 \let\g=\gamma
 \let\e=\varepsilon

 \let\lam=\lambda

 \let\f=\frac
 \let\inf = \infty
 \let \les = \lesssim
 \let\om=\omega
 
 \let \th = \theta

\let\B = \Big

 \let\Om=\Omega
 \let\td = \tilde

 \let\teq \triangleq
 
 \let\pa=\partial
 \let \bsh = \backslash

 \def\cA{{\mathcal A}}

 \def\cH{{\mathcal H}}

 \def\cL{{\mathcal L}}

 \def\cR{{\mathcal R}}
 
 \def\cT{{\mathcal T}}

 \def\la{\langle}
 \def\ra{\rangle}
\def\lt{\left}
\def\rt{\right}

 \newcommand{\beq}{\begin{equation}}
 \newcommand{\eeq}{\end{equation}}
  \newcommand{\bal}{\begin{aligned} }
  \newcommand{\eal}{\end{aligned}}
 \newcommand{\ben}{\begin{eqnarray}}
 \newcommand{\een}{\end{eqnarray}}
 \newcommand{\beno}{\begin{eqnarray*}}
 \newcommand{\eeno}{\end{eqnarray*}}

 \newcommand{\R}{\mathbb{R}}

\newcommand{\domega}{\boldsymbol{\omega}}

 \author{Jiajie Chen}
 \address{Applied and Computational Mathematics, California Institute of Technology, Pasadena, CA 91125, USA}
 \date{\today}

\title[ Slightly Perturbed De Gregorio Model]
{On the Slightly Perturbed De Gregorio Model on $S^1$}
 
\begin{document}

\begin{abstract}

It is conjectured that the generalization of the Constantin-Lax-Majda model (gCLM)
$\omega_t + a u\omega_x = u_x \omega$ due to Okamoto, Sakajo and Wunsch can develop a finite time singularity from smooth initial data for $a < 1$. For the endpoint case where $a$ is close to and less than $1$, we prove finite time asymptotically self-similar blowup of gCLM on a circle from a class of smooth initial data. For the gCLM on a circle with the same initial data, if the strength of advection $a$ is slightly larger than $1$, we prove that the solution exists globally with $|| \omega(t)||_{H^1}$ decaying in a rate of $O(t^{-1})$ for large time. The transition threshold between two different behaviors is $a=1$, which corresponds to the De Gregorio model.
\end{abstract}

 \maketitle

\section{Introduction}

Constantin, Lax and Majda \cite{CLM85} introduced an one-dimensional model (CLM)
\[
\om_t = u_x\om , \quad u_x = H \om
\]
to model the vortex stretching term in the three-dimensional Euler equations, where $H$ is the Hilbert transform. The advection term is missing in CLM model. In order to model both advection and vortex stretching, De Gregorio \cite{DG90,DG96} generalized the CLM model by adding 
an advection term $u \om_x$. By interpolating the CLM model and the De Gregorio model, Okamoto, Sakajo and Wunsch \cite{OSW08} further introduced the following one-parameter family of models (gCLM)
\beq\label{eq:DG}
\bal
\om_t + a u \om_x& = u_x \om , \quad u_x=  H  \om,
\eal
\eeq
where $a$ is a parameter and $H$ is the Hilbert transform. In the case of a circle, $H$ is given by 
\[
H \om (x) =    \f{1}{2 \pi} P.V. \int_{-\pi}^{\pi} \om(y) \cot( \f{x-y}{2}) dy.
\]
If $a = 0$, \eqref{eq:DG} is the CLM model. If $a=1$, it becomes the De Gregorio model. The gCLM shares some similarities with the 3D Euler equations. In fact, the 3D incompressible Euler equations in the vorticity formulation can be written as 
  \beq\label{eq:3DEuler}
    {\domega}_t+({\bf u}\cdot\nabla) {\domega} = \nabla
    {\bf u} \cdot {\domega} ,
    \eeq
  where ${\bf u}$ is the velocity field and ${\domega}=\nabla\times {\bf u}$ is the vorticity. 
The term $u_x \om$ in \eqref{eq:DG} models vortex stretching in \eqref{eq:3DEuler} which is the
main source of difficulty in obtaining global regularity of 3D Euler equations, and $u \om_x$ in \eqref{eq:DG} models advection in \eqref{eq:3DEuler} which has a stabilizing effect \cite{lei2009stabilizing,hou2008dynamic}. %Though the weight $a$ is missing in \eqref{eq:3DEuler}, recenly Elgindi \cite{elgindi2019finite} managed to create a small weight in the advection term in \eqref{eq:3DEuler} by studying certain class of solution.

The gCLM model \eqref{eq:DG} has been studied actively in recent years since it can characterize the competition between advection and vortex stretching in different scenarios. For $a < 0$, the advection would work together with the vortex stretching to produce a singularity.
% and finite time blowup is obtained by Castro and C\'ordoba \cite{Cor10} using a Lyapunov functional argument. 
Indeed, Castro and C\'ordoba \cite{Cor10}  proved the finite time blow-up for $a < 0$ based on a Lyapunov functional argument. The case of $a=0$ reduces to the CLM model and finite time singularity was established by Constantin, Lax and Majda \cite{CLM85}.

For $a > 0$, there is a competing nonlocal stabilizing effect due to the advection and a destabilizing effect due to vortex stretching. For small positive $a$, it is expected that the vortex stretching term will dominate the advection term. In \cite{Elg17}, Elgindi and Jeong constructed smooth self-similar profiles for small $a$ that lead to finite time blowup using a power series expansion and an iterative construction.
%The profile constructed in \cite{Elg17} decays slowly in the far field and the corresponding velocity field has infinite energy. 
In a recent joint work with Hou and Huang \cite{chen2019finite}, we established the stability of an approximate self-similar profile and obtained finite time asymptotically self-similar blowup for $C_c^{\infty}$ initial data. Similar results were obtained independently by Elgindi, Ghoul and Masmoudi \cite{Elg19} on the stability of the self-similar solutions constructed in \cite{Elg17} and the stability of the asymptotically self-similar blowup of \eqref{eq:DG}. For $a$ close to $\f{1}{2}$, where the vortex stretching term is relatively stronger, finite time asymptotically self-similar blowup for $C_c^{\infty}$ initial data has been established by the author \cite{chen2020singularity}. The self-similar singularities in these works are focusing in the sense that they can be written as $\om(x, t) = \f{1}{T-t} \Om( \f{x}{(T-t)^{c_l} })$ for some $T, c_l > 0$ and a nontrivial profile $\Om$.

For $a=1$, \eqref{eq:DG} reduces to the De Gregorio model. The analysis becomes much more challenging since advection and vortex stretching are comparable. Different behaviors of the solution of \eqref{eq:DG} on the real line and on a circle have been established. For \eqref{eq:DG} on a circle, in a remarkable work of Jia, Stewart and Sverak \cite{Sve19}, they established the nonlinear stability of the equilibria $A \sin( x - x_0)$ of \eqref{eq:DG} using spectral theories and complex variable methods. An alternative proof was obtained later by Lei, Liu and Ren \cite{lei2019constantin} using a direct energy method with a magic energy space. Moreover, global well-posedness of the solution has been obtained in \cite{lei2019constantin} for initial data $\om_0$ that has a fixed sign and $| \om_0|^{1/2} \in H^1$. Our analysis to be presented has benefited from some observations made in \cite{Sve19} and the stability analysis in \cite{lei2019constantin}. In particular, our stability analysis is built on the coercivity estimates of a linearized operator established in \cite{lei2019constantin} and we follow \cite{lei2019constantin} to establish the weighted $H^1$ estimates except that we need to estimate some additional perturbation terms. For \eqref{eq:DG} on the real line, in a recent joint work with Hou and Huang \cite{chen2019finite}, we proved finite time asymptotically self-similar blowup for $C_c^{\infty}$ initial data by establishing nonlinear stability of an approximate self-similar profile in the dynamic rescaling equation. In contrast to the case of small positive $a$ and $a$ close to $\f{1}{2}$, the self-similar singularity corresponding to $a=1$ is expanding and it can be written as $\om(x,t) = \f{1}{T-t} \Om( (T-t) x)$ for some $T >0$ and a nontrivial profile $\Om$. The strength of the advection plays an important role in the transition from focusing self-similar blowup to the expanding one.

For arbitrary large $|a|$, $C^{\al}$ self-similar profiles were constructed in \cite{Elg17}. Stability of approximate self-similar profiles and finite time asymptotically self-similar blowup for $C_c^{\al}$ initial data have been established in \cite{chen2019finite}. See also \cite{Elg19} for similar results. For this range of $a$, it is expected that advection dominates. Yet, an important observation made in \cite{Elg17} is that advection term can be substantially weakened by choosing $C^{\al}$ initial data with sufficiently small $\al$ so that the vortex stretching term still dominates. This observation and the idea of weakening the advection have played a crucial role in recent works of singularity formation, including Elgindi's breakthrough \cite{elgindi2019finite} on singularity formation of 3D axisymmetric Euler equations without swirl for $C^{1,\al}$ velocity and singularity formation of 2D Boussinesq equations and 3D axisymmetric Euler equations with $C^{1,\al}$ velocity and boundary obtained in my joint work with Hou \cite{chen2019finite2}. 
%have led to important progress of singularity formation of 3D Euler equations

Recently, Lushnikov, Silantyev and Siegel \cite{lushnikov2020collapse} performed analysis and extensive numerical study on \eqref{eq:DG} and provided numerical evidence for singularity formation of \eqref{eq:DG} with various $a$ and obtained a critical value $a_c \approx 0.6890665$. For \eqref{eq:DG} on the real line, they showed using well resolved computations that the expanding blowup for the case of $a=1$ obtained in \cite{chen2019finite} can be generalized to $ a_c < a \leq 1$. For \eqref{eq:DG} on a circle, they discovered a new type of self-similar blowup solutions of the form $\om(x, t) = \f{1}{t_c - t} f(x)$ for $a_c < a \leq 0.95$, which is neither focusing nor expanding. Due to some numerical difficulties, the range of $0.95 < a< 1$ was not explored and it was conjectured by the authors that self-similar blowup still exists. Our blowup results to be introduced are inspired by these new discoveries. There are other 1D models for the 3D Euler equations and the surface quasi-geostrophic equation, e.g. \cite{choi2014on,cordoba2005formation}, and we refer to \cite{Elg17} for an excellent survey.

In this paper, we study \eqref{eq:DG} on a circle with $a$ close to $1$, which can be regarded as a slightly perturbed De Gregorio model (\eqref{eq:DG} with $a=1$). For smooth initial data, this range of parameters $a$ is perhaps the most interesting one in \eqref{eq:DG} since it contains both the region where the strength of advection is slightly weaker than that of vortex stretching and the region where the strength of advection is slightly stronger than that of vortex stretching. We study finite time blowup of  \eqref{eq:DG} in the first region $a<1$ and the long time behavior of the solution of \eqref{eq:DG} in the second region $a> 1$.

\subsection{Main results}

We first introduce some weighted norms and spaces. 
\begin{definition}[Weighted norms and spaces]\label{def}
Define the singular weight $\rho =( \sin \f{x}{2})^{-2} $, the weighted norms $||\cdot ||_{\cH}, ||\cdot ||_X$ as follows
\beq\label{eq:XHnorm}
|| f||_{\cH}^2 \teq \f{1}{4\pi} \int_{-\pi}^{\pi} \f{ |f_{x}|^2}{\sin^2 \f{x}{2} } d x ,
\quad  || f ||^2_{ X } \teq || f||_{\cH}^2 + \int_{-\pi}^{\pi} |f_{x x }|^2 \cos^2 \f{x  }{2}  d x ,
\eeq
and the Hilbert spaces $\cH, X$  
\[
\cH \teq \{ f| f(0 )=0, ||f||_{\cH } < +\infty \} ,\quad X \teq \{ f | f(0) = 0 , ||f||_X <+\infty \}
\]
with inner products $\la \cdot, \cdot \ra_{\cH}, \la \cdot , \cdot \ra_X$ induced by the $\cH, X$ norm. 
\end{definition}

The $\cH$ norm and inner product $\la \cdot , \cdot \ra_{\cH}$ were first introduced in \cite{lei2019constantin} for stability analysis.
 %which plays an important role in the stability analysis.

Throughout this paper, we choose the gauge $u(t, 0) \equiv 0$ for the velocity field. The authors in \cite{Sve19} showed that solutions under different gauges are equivalent up to translations.
%a time-dependent translations. 

Our first main result is the existence of a family of self-similar solutions. 
\begin{thm}\label{thm:profile}
There exists an absolute small constant $\d_1$ such that for $1 - \d_1 < a < 1 + \d_1$, the gCLM 
\eqref{eq:DG} with parameter $a$ admits a self-similar solution  
\[
\om(x, t) = \f{1}{1 + c_{\om, a} t} \om_a(x) 
\] 
with an odd profile $\om_a$ and scaling parameter $c_{\om, a}$ satisfying 
\beq\label{eq:err_prof}
 ||  \om_a + \sin x  ||_X \les |1-a|, \quad | {c}_{\om,a} - (a-1) | \leq \min( C |1-a|^2, \f{1}{2} |1-a| )
 \eeq
for some absolute constant $C>0$. In particular, for $1-\d_1 < a < 1$, ${c}_{\om,a} < 0$  and $\om(x, t)$ blows up in finite time $ T = -\f{1}{ {c}_{\om,a}}$. For $a=1$, $ c_{\om, a} = 0$ and ${\om}_1 = -\sin x $. For $1 < a< 1 + \d_1$,  $ {c}_{\om, a} > 0$ and $\om(x, t)$ exists globally with $O(t^{-1})$ decay rate. 

Moreover, $\om_a \notin C^{\al}$ for $1< a< 1 + \d_1$ and any $\al(a) + 1 < \al < 1$, while $\om_a \in C^1$ but $\om_a \notin C^{1, \al }$ for $1-\d_1 <a< 1$ and any $ \al(a) < \al < 1$, where $\al(a)$ is given by
\beq\label{eq:hol_exp}
\al(a) = \f{  c_{\om,a} + (1-a) H \om_a( \pi) }{ a H\om_a(\pi)} 
\eeq
with $| H \om_a(\pi) +1| < \f{1}{10}$, $|\al(a) - 2(1-a) | \leq \min( \f{1}{2}|a-1|, C|a-1|^2)$ and $sign(\al(a)) = sign(1-a)$.

\end{thm}

The second main result addresses the stability of these profiles and the regularizing effect of the advection term.

\begin{thm}\label{thm:stability}
Let $(\om_a, c_{\om, a})$ be the profile and scaling in Theorem \ref{thm:profile}. There exist absolute small constants $\d_2, \d_3$ with $\d_2 < \d_1$, such that if $ \om_0$ is odd and satisfies 
$ || \om_0 + \sin x  ||_{\cH} < \d_3$, 
%and $ \int_{S^1} \om_0  d \th =0$, 
the following statements hold true for the solution $\om(x,t)$ of \eqref{eq:DG} with parameter $a$ and initial data $\om_0$.

(a) For $ 1 - \d_2 <  a < 1$, $\om(x, t)$ develops a singularity in finite time $T$ with lifespan $ T \geq \f{1}{ 2|1-a|}$. Moreover, $\om(x, t)$ is asymptotically self-similar and satisfies 
\[
|| \om( t) - \f{1}{\lam(t)}  \om_a ||_{\cH} \leq \lam(t)^{  \f{1}{4|a-1|} } || \om_0 - \om_a ||_{\cH}, 
\]
where $\lam(t) \leq 1$ is decreasing with $ \f{\lam(t) }{ T-t} \to - {c}_{\om, a} > 0 $ as $t \to T$. 

%(b) For $a=1$, $\om(x,t)$ exists globally with $ || \om(\cdot,t) + \sin x  ||_{\cH} \les e^{ - t/ 4} || \om_0 + \sin x ||_{\cH}$. 

(b) For $ 1 < a < 1 + \d_2$, $\om(x,t)$ exists globally, decays for large time and satisfies
\[ 
 || \om( t) - \f{1}{\lam(t)} {\om}_a ||_{\cH} \leq \lam(t)^{- \f{1}{4 |1-a|}}
 || \om_0 - \om_a ||_{\cH},
\]
where $\lam(t) \geq 1$ is increasing with $\f{ \lam(t)}{t} \to  {c}_{\om, a} > 0$ as $ t \to \infty$. Moreover, if $\om_0 \in H^s$ for $s > \f{3}{2}$, then for any $0 < \g < \g(a) $
\[
t^{ \g } |\om_{0,x}(\pi)|  \les_{a, \g} | \om_x(\pi, t)| \leq  || \om_x||_{L^{\infty}} 
\les || \om||_{H^s}
\]
for all $t > 0$, where $\g(a)= \f{ | (a-1) H \om_a( \pi)| }{ c_{\om, a}} $ with $|\g(a)-1|\les |a-1|$.

In particular, the solution of \eqref{eq:DG} with parameter $ 1-\d_2 < a < 1$ develops a singularity in finite time for some $C^{\infty}$ initial data.

\end{thm}

Theorem \ref{thm:stability} resolves the endpoint case of the conjecture made in \cite{lushnikov2020collapse,Elg17,okamoto2014steady} that the solution of gCLM \eqref{eq:DG} develops a finite time singularity for $a < 1$ from smooth initial data in the case of a circle. A related conjecture was stated in \cite{OSW08}. For \eqref{eq:DG} on the real line with $a=1$, finite time singularity with $C_c^{\infty}$ initial data has been established by Chen-Hou-Huang in \cite{chen2019finite}. 
%in a joint work with Hou and Huang \cite{chen2019finite}. 

The odd assumption on the initial data can be dropped with a price of losing the convergence estimates.

\begin{thm}\label{thm:weak}
There exist absolute small constants $\d_4, \d_5$, such that if $ || \om_0 + \sin x  ||_{\cH} < \d_5$ and $ \int_{S^1} \om_0  d \th =0$, the following statements hold true for the solution $\om(x,t)$ of \eqref{eq:DG} with parameter $a$ and initial data $\om_0$.

(a) For $1-\d_4 < a< 1$, $\om(x,t)$ develops a singularity in finite time $T$ with 
$|| \om(\cdot, t) + \lam(t)^{-1} \sin x ||_{\cH} < 2 \lam(t)^{-1} \d_5$, where $ \lam(t) \leq 1$ is decreasing and $ | \f{\lam(t) }{T-t} - |1-a| | \leq \f{|1-a|}{2} $.

%(b) For $a=1$, $\om(x,t)$ exists globally with $|| \om(x,t) + \sin x  ||_{\cH} \les e^{ - t/ 4} || \om_0 + \sin x ||_{\cH}$.

(b) For $ 1<a< 1+\d_4$,  $\om(x,t)$ exists globally and decays for large time with 
$
|| \om(\cdot, t) + \lam(t)^{-1} \sin x ||_{\cH} <2  \lam(t)^{-1} \d_5$
where $ \lam(t) \geq 1$ is increasing and $ | \f{\lam(t) }{t} - |1-a| | \leq \f{|1-a|}{2}$.
\end{thm}

\begin{remark}
For the case of $a=1$, stability results similar to those in Theorems \ref{thm:stability} and \ref{thm:weak} have been established in \cite{lei2019constantin,Sve19}. It was first proved in \cite{Sve19} that the equilibrium $-\sin x$ of \eqref{eq:DG} is nonlinear stable and the solution $\om(x,t)$ converges to $-\sin x$ exponentially fast in the $H^s$ norm for $s < \f{3}{2}$. Alternative proof of nonlinear stability was obtained later in \cite{lei2019constantin} and convergence in the $\cH$ norm was established : $|| \om(x, t) + \sin x ||_{\cH} \les e^{-\b t}|| \om_0(x, t) + \sin x ||_{\cH}$ for any $\b \in (0, \f{3}{8})$.  
%it was established in \cite{lei2019constantin} that for initial data $\om_0$ close to $-\sin x$ in the $\cH$ norm with $\int_{S^1} \om_0 dx = 0$, the solution exists globally with $|| \om(x,t) + \sin x  ||_{\cH} \les e^{ - \b t} || \om_0 + \sin x ||_{\cH}$ for any $\b \in (0, \f{3}{8})$. 
%The results (b) in Theorems \ref{thm:stability}, \ref{thm:weak} have been established in \cite{lei2019constantin}. We remark that for $a=1$, 
%Nonlinear stability of $-\sin x$ is first established in \cite{Sve19} in $H^s$ norm for $s<\f{3}{2}$. 
\end{remark}

\begin{remark}
Theorems \ref{thm:stability}, \ref{thm:weak} can be generalized to initial data close to $A \sin \th$ for $A > 0$, which can be established by applying Theorems \ref{thm:stability}, \ref{thm:weak} to $\om_A(x, t) \teq A \om(x, A t)$.
\end{remark}

\begin{remark}

The lifespan of the singular solution in result (a) in Theorem \ref{thm:stability} is at least $\f{1}{2(1-a)}$, which grows to $\infty$ as $a \to 1^-$. Thus it is hard to observe and resolve this singularity numerically.
%To the best of our knowledge, when $a$ is close to $1$, numerical evidence of such singular solution has not been reported in the literature.
\end{remark}

Result (b) in Theorem \ref{thm:stability} shows that the singularity developed from the solution of \eqref{eq:DG} with $a$ less than $1$ can be regularized as the strength of advection increases.
We remark that the regularizing effect of advection has also been studied by Hou-Li \cite{hou2008dynamic} and Hou-Lei \cite{lei2009stabilizing} for the 3D axisymmetric Navier-Stokes equations.
%, which confirms the regularizing effect of advection studied by Hou-Lei \cite{lei2009stabilizing} and Hou-Li \cite{hou2008dynamic}. 
Result (b) also confirms the numerical evidence in \cite{lushnikov2020collapse} that the solution of \eqref{eq:DG} with $a$ slightly larger than $1$ exists globally for some initial data with decay of $||\om||_{L^{\infty}}$ and unbounded growth of $||\om_x||_{L^{\infty}}$ as $t \to \infty$. In addition, it relates to small scale creation in the solution which can be measured in the $H^s$ norm for $s > \f{3}{2}$. A similar observation was made in \cite{Sve19} for the solution of \eqref{eq:DG} with $a=1$.

The qualitative behavior of the solution of \eqref{eq:DG} for various $a$ in Theorems \ref{thm:stability}, \ref{thm:weak} can be characterized by the following simple ODE
\[
 \f{d}{dt} f(t) = (1-a) f(t)^2, \quad f(0) = 1,
\]
which has a solution $f(t) = \f{1}{ 1 - (1-a) t}$. For $a<1$, $f(t) $ blows up in finite time $T = \f{1}{1-a}$. For $a>1$, $f(t)$ exists globally with $O(t^{-1})$ decay rate. %This characterization is inspired by an observation of V. Sverak (personal communications, August 5, 2019) that a self-similar blowup solution of \eqref{eq:DG} that is neither focusing nor expanding could exist for $a$ close to $0.7$, and the behavior of the solution is similar to the behavior of the solution of some ODE.

% In \cite{Sve19}, for the solution $\om(x,t)$ of \eqref{eq:DG} with $a=1$ and initial data close to $-\sin \th$, it is pointed out that the higher order $H^s, s > \f{3}{2}$ norm of the perturbation $|| \om + \sin \th||_{h^s}, s > \f{3}{2}$ typically grow exponentially as $t \to \infty$. This obervation indicates that some small scale structure is created in the solution and can be measured using some norm at the level of $H^s$ for $s > \f{3}{2}$. It also applies to the solution of \eqref{eq:DG} with $a$ slightly larger than $1$ and result (c) in Theorem \ref{thm:stability} shows small scale creation in the solution.

\begin{remark}
As we will see, the profile $\om_a$ is not smooth near $x = \pi$. We impose the odd assumption on the initial data in Theorem \ref{thm:stability} so that we can estimate $ u\om_{a,xx}$ using the vanishing condition $u(\pi) = 0$ that compensates the nonsmoothness of $\om_a$. 
 %As we will see, the profile $\om_a$ is not smooth near $x = \pi$. We impose the odd assumption on the initial data in Theorem \ref{thm:stability} so that we can estimate $u \om_{a,xx}$ in some weighted $L^2$ space using the vanishing condition $u(\pi) = 0$ due to the odd assumption.
\end{remark}

\begin{remark}
Local well-posedness of \eqref{eq:DG} in $H^s$ for any $s > \f{3}{2}$ can be proved by the Kato-Ponce commutator estimate \cite{kato1988commutator} and energy estimate. We will use the dynamic rescaling formulation discussed in Section \ref{sec:dsform} to perform a-priori estimate on the rescaled perturbation $\om(x,t) - \f{1}{\lam(t)}  \om_a$, which implies that the perturbation remains in the functional space $\cH$ locally in time.
\end{remark}

\subsection{The main ideas and the outline of the proofs}
\subsubsection{ Competition between advection and vortex stretching }
The fundamental idea in the proofs is to study the evolution of \eqref{eq:DG} using the dynamic rescaling equation 
\beq\label{eq:DG_dy0}
\om_t + a u \om_x = (c_{\om} + u_x) \om
\eeq
with the normalization condition $c_{\om}(t) = (a-1)u_x(t, 0)$. Formally $a-1$ characterizes the relative strength of the advection term $a u \om_x$ and the vortex stretching term $u_x \om$. A crucial observation is that if $u_x(t, 0)$ has a fixed sign for all $t > 0$, then the sign of  $c_{\om}$ captures the competition between these two terms.
% A crucial observation is that if $\om$ is close to some nontrivial profile $\bar \om$ with $\bar u_x(0) = H \bar \om(0) \neq 0$ for all time, then $|u_x(0)|$ is bounded away from $0$ and 
% $u_x(0)$ is close to $\bar u_x(0)$ and has a fixed sign. Hence, the sign of $c_{\om}$ is completely determined by $a-1$. Formally $a-1$ characterizes the relative strength of the convection term $a u \om_x$ and the vortex strecthing term $u_x \om$. The competition between these two terms is encoded in $c_{\om}$. 
In particular, if the solution $\om(x, t)$ of \eqref{eq:DG_dy0} is close to a nontrivial profile and $u_x(0) \geq c >0$ for all $t >0$, then the sign of $c_{\om}$ determines the long time behavior of the solution $\om^{phy}$ of \eqref{eq:DG} after rescaling the solution $\om(x, t)$. In particular, we can prove :

\begin{itemize}
	\item  If $c_{\om}(t) \leq -|1-a| c < 0$ for all $t > 0$, $\om^{phy}$ develops a singularity in finite time.

\item  If $c_{\om}(t)=0$ for all $t >0$, $\om^{phy}$ remains close to the profile for all time. 

\item If $c_{\om}(t) \geq |1-a|c > 0$ for all $t > 0$, $\om^{phy}$ exists globally and decays for large time. 
\end{itemize}

See more discussion on the dynamic rescaling formulation in Section \ref{sec:dsform}.

\subsubsection{Outline of the proofs}
We close the introduction by sketching the steps in the proofs.

% 1. Reformulate \eqref{eq:DG} using the dynamic rescaling formulation and construct an approximate steady state (approximate blowup profile). To our surprise, the nonlinear stable equilibrium $ -\sin x$ of equation \eqref{eq:DG} with $a=1$ \cite{Sve19,lei2019constantin} provides a good approximate blowup profile.

1. Reformulate \eqref{eq:DG} using the dynamic rescaling formulation.

2. We follow the method of analysis in \cite{chen2019finite} to construct the self-similar profile $ \om_a$ and scaling $ c_{\om, a}$ of \eqref{eq:DG} for $a$ close to $1$. There are several steps in this method. Firstly, construct an approximate steady state of the dynamic rescaling equation (approximate self-similar profile). To our surprise, the equilibrium $ -\sin x$ of equation \eqref{eq:DG} with $a=1$ provides a good approximate self-similar profile. Secondly, perform nonlinear stability analysis of the approximate steady state in some suitable weighted Sobolev norm. We will use the coercivity estimates of a linearized operator established in \cite{lei2019constantin} to perform stability analysis in the weighted $H^1$ space $\cH$. Then we further establish stability analysis in the weighted $H^2$ space $X$ using the energy method. We remark that some weighted Sobolev spaces with singular weights have been used in \cite{Sve19,lei2019constantin,chen2020singularity,Elg19,chen2019finite} for the nonlinear stability analysis of \eqref{eq:DG}. Finally, establish convergence of the solution of \eqref{eq:DG_dy0} to a self-similar profile by time differentiation. Similar time-differentiation arguments have been developed in \cite{elgindi2019finite,chen2019finite}.

%2. We use the method developed in \cite{chen2019finite} to construct the self-similar profile $ \om_a$ and scaling $ c_{\om, a}$ of \eqref{eq:DG} for $a$ close to $1$, in particular the idea of constructing an approximate steady state and the argument of establishing convergence to a self-similar profile by time differentiation. We use the coercivity estimates of a linearized operator established in \cite{lei2019constantin} to perform stability analysis of the approximate steady state in the weighted $H^1$ space $\cH$. We further establish stability analysis in the weighted $H^2$ space $X$ using energy method.

3. Establish nonlinear stability analysis of the profile $ \om_a$ in the weighted $H^1$ space $\cH$. Obtain the results about \eqref{eq:DG} with various $a$ by rescaling the solution. 

%This enables us to obtain a stability region with uniform radius for the cases of different $a$, i.e. $\d_3$ in Theorem \ref{thm:stability}.

\vspace{0.1in}
\paragraph{\bf{Organization of the paper}}
In Section \ref{sec:profile}, we construct a family of self-similar profiles of \eqref{eq:DG} for $a$ close to $1$. In Section \ref{sec:stability}, we study the stability of these profiles and establish the convergence estimates in Theorem \ref{thm:stability}. In Section \ref{sec:weak}, we prove Theorem \ref{thm:weak}. We discuss an approach which has the potential to be applied to obtain finite time blowup for \eqref{eq:DG} with other $a<1$ in Section \ref{sec:approach}. In the Appendix, we prove some properties of the Hilbert transform and a functional inequality.

\vspace{0.1in}
\paragraph{\bf Notations}

We use $\la \cdot , \cdot \ra$ to denote the standard inner product on $S^1$:
%\beq\label{nota:inner}
$\la f , g \ra  \teq \int_{S^1} f g dx .$
%\eeq
 %We can use the weight $\rho$ in Definition \ref{def} and the inner product to simplify the notation. 
 Recall $\cH$ in Definition \ref{def}. For $f, g \in \cH$, we have  
 \beq\label{nota:innerH}
\la f, g \ra_{\cH} = (4\pi)^{-1} \la f_x, g_x \rho \ra .
\eeq 

Without specification, $\int f dx$ means $\int_{S^1} f dx $. We use $C, C_i$ to denote absolute constants and $C(A,B,..,Z)$ to denote constant depending on $A,B ,..,Z$. These constants may vary from line to line, unless specified. We also use the notation $A \les B (A \gtrsim B)$ if there is some absolute constant $C$ such that $A \leq CB (A \geq CB)$.

\section{Construction of the self-similar profiles}\label{sec:profile}

In this Section, we prove Theorem \ref{thm:profile} using the strategy in \cite{chen2019finite}. We first perform weighted $H^2$ stability analysis of an approximate steady state in the dynamic rescaling equation and then establish convergence to the exact steady state, from which we can obtain the self-similar profile. %The key part in this Section is the weighted $H^2$ estimate. % that enables us to further prove convergence. 

\subsection{Dynamic rescaling formulation}\label{sec:dsform}

Let $ \om(x, t), u(x,t)$ be the solutions of equation \eqref{eq:DG}. It is easy to show that 
\beq\label{eq:rescal1}
  \td{\om}(x, \tau) = C_{\om}(\tau) \om(   x,  t(\tau) ), \quad   \td{u}(x, \tau) = C_{\om}(\tau) 
u( x, t(\tau))
\eeq
are the solutions to the dynamic rescaling equations
 \beq\label{eq:rescal2}
\bal
\td{\om}_{\tau}(x, \tau) +  a \td{u}  \td{\om}_x(x, \tau)  &=   c_{\om}(\tau) \td{\om} + \td{u}_x \td{ \om  } , \quad \td{u}_x &= H \td{\om} ,
\eal
\eeq
where  
\beq\label{eq:rescal3}
\bal
  C_{\om}(\tau) = \exp\lt( \int_0^{\tau} c_{\om} (s)  d \tau\rt) C_{\om}(0), \quad   t(\tau) = \int_0^{\tau} C_{\om}(\tau) d\tau.
\eal
\eeq
We have the freedom to choose the initial rescaling factors $ C_{\om}(0)$ and impose some normalization condition on the time-dependent scaling parameter $c_{\om}(\tau)$. Then the equation \eqref{eq:rescal2} is completely determined and the solution of \eqref{eq:rescal2} is equivalent to that of the original equation \eqref{eq:DG} using the scaling relationship given in \eqref{eq:rescal1}-\eqref{eq:rescal3}, as long as $c_{\om}(\tau)$ remain finite.

\begin{remark}
In \eqref{eq:rescal1}-\eqref{eq:rescal2}, we do not rescale the spatial variable. This is different from the dynamic rescaling equation in \cite{chen2019finite} which contains a factor $C_l(\tau)$ in \eqref{eq:rescal1} and a stretching term $c_l( \tau)x \om_x$ in \eqref{eq:rescal2}. Here, we simply choose $c_l(\tau) \equiv 0$ and $C_l(\tau) \equiv 1$.
\end{remark}

We remark that a similar dynamic rescaling formulation was employed in \cite{mclaughlin1986focusing,  landman1988rate} to study the nonlinear Schr\"odinger (and related) equation. In some literature, this formulation is called the modulation technique. It has been a very effective tool to study singularity formation for many problems like the nonlinear Schr\"odinger equation \cite{kenig2006global,merle2005blow}, the nonlinear wave equation \cite{merle2015stability}, the nonlinear heat equation \cite{merle1997stability}, the generalized KdV equation \cite{martel2014blow}, and other dispersive problems. 

Suppose $C_{\om}(0)=1$. If $c_{\om}(\tau) \leq -C <0$ for some $C>0$ and any $\tau > 0$ and the solution $\td{\om}$ is nontrivial, e.g. $ || \td{\om}(\tau, \cdot) ||_{L^{\infty}} \geq c >0$ for all $\tau >0$, we then have 
\[
C_{\om}(\tau) \leq e^{-C\tau}, \quad t(\infty) \leq \int_0^{\infty}  e^{-C \tau } d \tau =C^{-1} <+ \infty 
\]
and that 
\[
| \om( x,  t(\tau) ) | = C_{\om}(\tau)^{-1}  |\td{\om}(x, \tau) | \geq e^{C\tau} |\td{\om}(x, \tau) | 
\]
blows up at finite time $T = t(\infty)$. 

On the contrary, if $c_{\om}(\tau) \geq C >0$ for some $C>0$ and $\td \om(x, \tau)$ is bounded, e.g. $|| \td \om(x, \tau)||_{L^{\infty}} \leq c$, for any $\tau >0$, then we obtain 
\[
C_{\om}(\tau) \geq e^{C\tau}, \quad t(\tau) \geq \int_0^{\tau}  e^{C\tau}  d \tau
\]
and that
\[
| \om( x,  t(\tau) ) | = C_{\om}(\tau)^{-1}  |\td{\om}(x, \tau) | \leq e^{ - C\tau} | \td \om(x,\tau) |\leq e^{-C\tau} c
\] 
decays for large $\tau$. Due to the fact that $t(\tau) \to \infty$ as $\tau \to \infty$ and the above estimate on $\om$, we can obtain global existence of solution using Beale-Kato-Majda-type criterion.

If $(\td{\om}_{\tau}, c_{\om}(\tau))$ converges to a steady state $(\om_{\infty}, c_{\om,\infty})$ of \eqref{eq:rescal2} as $\tau \to \infty$, one can verify that 
\beq\label{eq:relation}
\om(x, t) = \f{1}{1 +  c_{\om,\infty} t} \om_{\infty}(x)
\eeq
is a self-similar solution of \eqref{eq:DG}. Due to this connection, we do not distinguish the steady state of \eqref{eq:rescal2} and the self-similar profile of \eqref{eq:DG}.

To simplify our presentation, we still use $t$ to denote the rescaled time in the rest of the paper, unless specified, and drop $\td{\cdot}$ in \eqref{eq:rescal2}, which leads to
\beq\label{eq:DGdy}
\bal
\om_t +  au \om_x &= (c_{\om} + u_x) \om  , \quad u_x = H\om .
\eal
\eeq

Using integration by parts and the antisymmetry property of the Hilbert transform, we have
\[
\int_{S^1} u_x \om  - a u \om_x  d x = (1 + a) \int_{S^1} \om \cdot H \om dx = 0.
\]
Therefore, if $\int_{S^1}\om_0 dx = 0$, this property is conserved 
\beq\label{eq:conserve}
 \f{d}{dt }\int_{S^1} \om d x = c_{\om} \int_{S^1} \om d x = 0.
\eeq

\subsection{ Approximate steady state and linearization}

% For $a$ close to  Surprisingly, the ground state of \eqref{eq:DG} with $a=1$, i.e. $\om(x) = -\sin \th$, which is proved to be nonlinear stable \cite{Sve19,lei2019constantin}, provides a good approximate steady state of \eqref{eq:DGdy} in the study of 

We use the equilibrium $-\sin x$ of \eqref{eq:DG} with $a=1$ to construct the approximate steady states for \eqref{eq:DGdy}
\beq\label{eq:profile}
\bar \om = -\sin x , \quad \bar u = \sin x, \quad \bar c_{\om} = (a-1) \bar u_x(0) =  a-1.
\eeq
We impose the following normalization condition on $c_{\om}$ in \eqref{eq:DGdy} 
\beq\label{eq:normal}
c_{\om} = (a-1) u_x(0). 
\eeq

Linearizing \eqref{eq:DGdy} around the above approximate steady state, we can obtain the equation for the perturbation $(\om, c_{\om})$ ($\om + \bar \om, c_{\om} + \bar c_{\om}$ is the solution of \eqref{eq:DGdy})
\beq\label{eq:lin}
\bal
\om_t & = - a \bar u \om_x + (\bar c_{\om} + \bar u_x) \om + (c_{\om} + u_x ) \bar \om - a u \bar \om_x + N( \om ) + F(\bar \om)  \\
&\teq \cL_a \om +   N(\om ) + F(\bar \om),
\eal
\eeq
where $\cL_a$ denotes the linearized operator and the nonlinear term and error term are given below 
\beq\label{eq:NF}
\bal
N(\om) = (c_{\om} + u_x ) \om - a u \om_x , \quad F(\bar \om) = 
( \bar c_{\om} + \bar u_x ) \bar \om - a \bar u  \bar\om_x.
\eal
\eeq

%For $a=1$, p
Plugging the approximate steady state \eqref{eq:lin} and the normalization condition into \eqref{eq:lin} yields 
\beq\label{eq:linop}
\bal
\cL_1 \om & = -\sin x \om_x +  \cos  x \om - u_x \sin x + u \cos x  \\
\cL_a \om & = -a \sin x \om_x + (a-1 + \cos x ) \om - ( (a-1) u_x(0) + u_x)  \sin x  + a u \cos x  \\
& = \cL_1 \om + (a-1) ( - \sin x \om_x +  \om - u_x(0) \sin x + u \cos x )= \cL_1 \om  + (a-1) \cA \om ,\\ 
\eal
\eeq
where $\cA$ is given by
\beq\label{eq:linA}
\cA \om = - \sin x \om_x +  \om - u_x(0) \sin x + u \cos x.
\eeq

We consider initial perturbation $\om_0 \in X$ with $\int_{S^1} \om_0 dx =0 $, where $X$ is defined in \eqref{eq:XHnorm}. Using $\bar \om = -\sin x$ and \eqref{eq:conserve}, we yield 
\[
 0 =\int \om_0(x) + \bar \om(x) dx  = \int \om(x, t) + \bar \om(x) dx .
\]
Thus the perturbation satisfies  $\int_{S^1} \om(x, t) dx = 0$ for $t>0$.
% Since $\int_{S^1} \bar \om d x = -\int_{S^1} \sin x dx =  0$ and 
% $ \int_{S^1} \bar  \om + \om(x, t) dx$ is conserved \eqref{eq:conserve}, we obtain $\int_{S^1} \om(x, t) dx = 0$ for the perturbation.

Recall the $\cH$ norm in \eqref{eq:XHnorm} and the operator $\cL_1$ in \eqref{eq:linop}. Our stability analysis is built on the work of Lei et. al. \cite{lei2019constantin}, in which the authors proved the following results.

\begin{lem}\label{lem:linop}
Suppose that $ f, g \in \cH$ and $\int_{S^1} f dx = 0$. Denote $e_0( x) = \cos x -1$ and 
\[
 f_e = \la f, e_0 \ra_{\cH}  , \quad  \la f, g \ra_Y  \teq \la f - f_e e_0, g - g_e e_0 \ra_{\cH}.
\]
We have :
(a) Equivalence of norms : $( \cH / \R \cdot e_0, \la \cdot , \cdot \ra_Y)$ is a Hilbert space and the induced norm $|| \cdot ||_Y$ satisfies 
\[
\f{1}{2} || f ||_{\cH} \leq  || f ||_{Y} \leq || f ||_{\cH} .
\]
%for $ f \in \cH$ with $\int_{S^1} f dx = 0$. 

(b) Orthogonality : $||e_0||_{\cH} = 1$ and 
\[
\la f -  f_e e_0 , e_0 \ra_{\cH} = 0, \quad || f ||^2_{\cH} = f_e^2 + || f||_Y^2.
\]
 %$\la f -  f_e e_0 , e_0 \ra_{\cH} = 0$, $ || e_0||_{\cH} = 1$ and $|| f ||^2_{\cH} = f_e^2 + || f||_Y^2$.

(c) Coercivity : $\la \cL_1 f , f \ra_{Y} \leq -\f{3}{8} ||f ||^2_Y$. 
\end{lem}

\begin{remark} The constant $\f{1}{2}$ in (a) is a direct consequence of the result in \cite{lei2019constantin}, which implies $ || f||_{\cH}^2 \leq || f||_Y^2 ( 1 + \sum_{k\geq1} \f{1}{ k^2 (k+1)^2} )  < 4 || f||_Y^2$ by using the Cauchy-Schwarz inequality. %Clearly, the constant in the upper bound is bounded by $ 4$.
\end{remark}

For $f, g \in \cH$ with $\int_{S^1} f dx = 0$, results (a), (b) in the above Lemma implies
\beq\label{eq:YH}
 \la f, g \ra_Y  = \la f , g - g_e e_0 \ra_{\cH} = \f{1}{4\pi} \la f_x, (g_x + g_e \sin x) \rho \ra ,
 %= \f{1}{4\pi} \int_{S^1} \f{f_x (g_x + g_e \sin x) }{ \sin^2 \f{x}{2}} dx  
 %, \quad |\la f, g \ra_Y | \leq || f||_{\cH} || g||_{\cH},
\eeq
where we have used $\pa_x e_0 =- \sin x$ and the notation \eqref{nota:innerH}.

 The following simple integration by parts will be used repeated.

\begin{lem}\label{lem:IBP}
Let $f \in L^2( \rho )$. We have 
\[
\la \sin x f_x, f \rho \ra = \f{1}{2} \la f^2, \rho \ra .
\]
% Let $f \in L^2( \sin^{-2} \f{x}{2})$. We have 
% \[
% \int_{S^1} \f{ \sin x f_x f}{\sin^2 \f{x}{2}} dx  = \f{1}{2} \int_{S^1} \f{ f^2}{\sin^2 \f{x}{2}} dx.
% \]
\end{lem}

The proof is straightforward and omitted. 
\begin{remark}
Formally, the above identity can be interpreted as that $ \sin^{-2} \f{x}{2}$ is an eigenfunction of the adjoint of $ \sin x \pa_x$ with eigenvalue $1$. %Its analog on the real line is that 
An analog of it is that on the real line, $x^{-k}$ is an eigenfunction of the adjoint of $x \pa_x $ with eigenvalue $k-1$, which plays an important role in the stability analysis in \cite{chen2019finite}. It seems that the mysterious inner product $\la \cdot ,\cdot \ra_{\cH}$ and singular weight $\sin^{-2} \f{x}{2}$ constructed in \cite{lei2019constantin} arise naturally from the viewpoint of energy estimates. See more related estimates in Lemma \ref{lem:tri} and Section \ref{sec:H2}.
\end{remark}

In the following discussion, we will first establish weighted $H^1$ and weighted $H^2$ estimates of the linearized equation and then control the remaining nonlinear and error terms. 

% \begin{remark}
% The weighted $H^1$ estimates does not require the odd condition on the initial data. We will impose this condition in the weighted $H^2$ estimate of the nonlinear term.
% \end{remark}

\subsection{Weighted $H^1$ estimates}\label{sec:H1}
Recall that the perturbation satisfies $ \int_{S^1} \om dx = 0$. Performing energy estimate on $\la \om, \om\ra_Y$ yields
\beq\label{eq:H1_est1}
\f{1}{2}\f{d}{dt} \la \om, \om \ra_Y = \la \cL_a \om , \om \ra_Y + \la N(\om) , \om \ra_Y
+ \la F(\bar \om), \om \ra_Y .
\eeq
Recall the operators $\cL_1, \cL_a$ in \eqref{eq:linop} and $\cA$ in \eqref{eq:linA}.
% . Denote 
% \[
% \cA \om = - \sin x \om_x +  \om - u_x(0) \sin x + u \cos x.
% \]
A direct calculation yields 
\[
\bal
\pa_x \cA \om =  -\sin x \om_{xx} + \om_x (1 -\cos x ) + (u_x -u_x(0)) \cos x - u \sin x \teq - \sin x \om_{xx} + \cA_1 \om .
\eal
\]

Applying Lemma \ref{lem:linop} and the identities in \eqref{eq:linop}, \eqref{eq:YH}, we derive 
%\beq\label{eq:lin1}
\[
\bal
\la \cL_a \om, \om \ra_Y &= \la \cL_1 \om + (a-1) \cA \om, \om \ra_Y \\
&\leq -\f{3}{8} ||\om||_Y^2 
+ \f{a-1}{4\pi} \int_{S^1} \f{ \pa_x \cA \om \cdot (\om_x + \om_e \sin x) }{ \sin^2 \f{x}{2}} dx .
%+ \f{a-1}{4\pi}  \la \pa_x \cA \om ,  (\om_x + \om_e \sin x) \rho \ra.
\eal
\]
%\eeq

Recall $\rho = (\sin \f{x}{2})^{-2}$. Using the Hardy-type inequality, the Poincar\'e inequality, the Cauchy-Schwarz inequality and the isometry property of the Hilbert transform, we have 
\beq\label{eq:basic1}
\bal
&\la (u_x -u_x(0))^2, \rho \ra \les || u_{xx} ||_{L^2} \les  || \om_x||_{L^2} \les || \om ||_{\cH} , \\ 
% &\int_{S^1} \f{ (u_x -u_x(0))^2}{ \sin^2 \f{x}{2}} dx 
% \les \int_{S^1} u^2_{xx} dx \les || \om_x||_{L^2} \les || \om ||_{\cH} , \\
& ||u ||_{L^2} \les || \om||_{\cH} , \qquad |\om_e| = |\la \om, e_0 \ra_{\cH} | \les || \om||_{\cH}.
\eal
\eeq

We remark that the Hardy-type inequality used above can be proved by applying an integration by parts argument 
%to $\int_{S^1} (u_x - u_x(0))^2  \pa_x \cot \f{x}{2} dx$ 
and thus we omit the proof. Notice that 
$ |1-\cos x| \les |\sin \f{x}{2}|$ and $|\sin x | \les |\sin \f{x}{2}|$. We obtain 
\[
\B|\int_{S^1} \f{\cA_1 \om \cdot (\om_x + \om_e \sin x) }{\sin^2 \f{x}{2}} dx \B| \les  || \om||^2_{\cH}.
\]

Using Lemma \ref{lem:IBP}, we obtain 
% \[
% I \teq \la  -\sin x \om_{xx}, (\om_x + \om_e \sin x) \rho \ra 
% = -\f{1}{2} \la \om_x^2 ,\rho \ra -4 \om_e \la \om_{xx} \cos^2 \f{x}{2} \ra.
% \]
\[
I \teq \int_{S^1} \f{ -\sin x \om_{xx} \cdot (\om_x + \om_e \sin x) }{ \sin^2 \f{x}{2}} dx
= -\f{1}{2} \int_{S^1} \f{\om_x^2}{ \sin^2 \f{x}{2}} dx 
-4 \om_e \int_{S^1} \om_{xx} \cos^2 \f{x}{2} dx .
\]
Applying integration by parts and the estimate \eqref{eq:basic1} on $\om_e$ yields
%using Lemma \ref{lem:linop} (or simply the Cauchy-Schwarz inequality) to control $f_e$ yields
\[
|I| \les || \om ||^2_{\cH} + | \om_e |  || \om_x ||_{L^1} \les || \om||^2_{\cH}.
\]

Combining the above estimates, we establish 
\beq\label{eq:H1_est2}
\la \cL_a \om, \om \ra_Y \leq -\f{3}{8} || \om||_Y^2 + C|a-1|  || \om ||^2_{\cH}.
\eeq

\subsection{Weighted $H^2$ estimates}\label{sec:H2}

The weighted $H^2$ estimates is not necessary in order to obtain finite time blowup of \eqref{eq:DG} with $a$ less than and close to $1$ from some smooth initial data since the nonlinear estimate can be closed once the nonlinear and error terms are estimated 
%(see Section \ref{sec:NF}) 
using the energy $|| \om||_Y$. %With the weighted $H^1$ estimate of nonlinear and error terms in Section \ref{sec:NF}, one can use a bootstrap argument to obtain nonliear stability of the perturbation, which can further imply finite time blowup by recaling the solution of \eqref{eq:DGdy} back to the solution of \eqref{eq:DG}. 
This is done in Section \ref{sec:weak}. See also \cite{chen2019finite} for the argument to establish blowup. Yet, we cannot determine if the solution of \eqref{eq:DGdy} converges as $t \to \infty$ or if the singularity is asymptotically self-similar. To further obtain convergence, we apply the time-differentiation argument in \cite{chen2019finite}. %Due to the fact that 
Since $F(\bar \om)$ is time-independent, taking $\pa_t$ in \eqref{eq:lin}, we derive 
\beq\label{eq:time}
\pa_t (\om_t) = \cL_a \om_t + \pa_t N(\om) .
\eeq
There are two approaches to estimate $\om_t$. 

(a) Estimate $\om_t$ in some norm that is weaker than the $Y$ norm. 

(b) Further perform estimate on $\om$ in some energy norm that is stronger than the $Y$ norm and then estimate $\om_t$ in the $Y$ norm. 

In \cite{chen2019finite}, the first approach is applied and $\om_t$ is estimated in some weighted $L^2$ space. Here, we do not have analysis of $\cL_a$ in some weaker norm. Alternatively, we apply the second approach which is simpler since we can use the a-priori weighted $H^1$ estimate established in Section \ref{sec:H1}. It is pointed out in \cite{Sve19} that in the case of $a=1$, where $c_{\om}, \bar c_{\om} = 0$ and \eqref{eq:DGdy} reduces to \eqref{eq:DG}, the $H^s,s > \f{3}{2}$ norm of the perturbation around the steady state $-\sin x$ typically grows exponentially. We expect similar instability for the perturbation around the approximate steady state \eqref{eq:profile} of \eqref{eq:DGdy} since $a$ is close to $1$. Thus the higher order estimate is nontrivial and we need to design the weighted $H^2$ norm carefully. 

We introduce the weighted derivative $D_x = \sin x \pa_x$. Similar weighted derivative has been used in \cite{elgindi2019finite,chen2019finite2,Elg19} for stability analysis. We have the following commutator estimate. 
\begin{lem}\label{lem:commu}
For $f, D_x f \in L^2$, 
\[
[D_x , H] f =D_x H f - H( D_x f ) = \f{1}{2\pi} \la \om, \sin x \ra .
\]
\end{lem}

We defer the proof to the appendix. 

Taking $D_x$ on both sides of \eqref{eq:lin} yields 
\beq\label{eq:H2}
 \pa_t D_x \om = D_x \cL_a \om + D_x N(\om) + D_x F(\bar \om).
\eeq
%We first derive the leading order term in $\pa_x \cL_a \om$ and then 

Our goal is to perform weighted $H^1$ estimate of $D_x \om$ with weight  $\rho = (\sin \f{x}{2})^{-2}$. To simplify the derivation, we use $l.o.t.$ (lower order terms) to denote the terms whose weighted $L^2(\rho)$ norm can be bounded by $C || \om ||_{\cH}$ for some absolute constant $C$.  
It can vary from line to line. Using the Hardy-type inequality and the $L^2$ isometry of the Hilbert transform, we obtain 
\beq\label{eq:L2_xt1}
\bal
&|| u \rho^{1/2}||_{L^2} \les || u_x||_{L^2} \les || \om ||_{\cH}, \quad || u_{xx}||_{L^2} \les || \om ||_{\cH}, \\
& || \om \rho^{1/2} ||_{L^2} \les || \om_x||_{L^2} \les || \om ||_{\cH}.
\eal
\eeq
Combining the above estimates and \eqref{eq:basic1}, we yield that
%\beq\label{eq:lot}
\[
u, \quad  u_x - u_x(0), \quad  \sin x \cdot u_x,   \quad \sin x \cdot u_{xx}  = D_x u_x, \quad  \om, \quad \om_x 
, \quad \om_e \sin x 
\]
and the product of these terms and smooth functions are l.o.t.

 Recall $\cL_a$ in \eqref{eq:linop}. We rewrite $\cL_a$ as follows 
\[
\bal
\cL_a &= \B( -a D_x \om + a \om \cos x  \B)+ (a-1) (1 - \cos x ) \om   + \B( - u_x \sin x + u \cos x \B) \\
&+ (a- 1) ( u \cos x - u_x(0) \sin x  ) \teq P_1 + P_2 + P_3 + P_4.
\eal
\]

Note that $D_x$ satisfies the Leibniz rule 
\[D_x(fg) = gD_x f + f D_x g .\]

Using $\pa_x(\om D_x \cos x) = l.o.t.$ and a direct calculation yields 
\[
\bal
\pa_x D_x P_1 &= a  \pa_x ( - \sin x \pa_x D_x \om + \cos x D_x \om + \om D_x \cos x ) \\
&=  a(  - \sin x \pa_x \pa_x D_x \om  - \cos x \pa_x D_x \om + \cos x \pa_x D_x \om
- \sin x D_x \om) + l.o.t. \\
& = - a \sin x \pa_x \pa_x D_x \om + l.o.t. .
\eal
\]
For $P_3$, using $\pa_x( u \sin^2 x) =l.o.t.$ and Lemma \ref{lem:commu}, we get
\[
\bal
\pa_x D_x P_3 & = \pa_x( \sin x ( \pa_x P_3))  =\pa_x (  \sin x ( -u_{xx} \sin x - u_x \cos x + u_x \cos x - u \sin x  ) ) \\
& =  -\pa_x( \sin x D_x u_x) + l.o.t.    =  -\cos x  D_x u_x  - \sin x \pa_x ( D_x u_x ) + l.o.t.  \\
&=- \sin x \pa_x ( D_x u_x ) + l.o.t.  =  - \sin x  \pa_x \B(  H( D_x \om) +  \f{1}{2\pi} \la \om, \sin x \ra \B) + l.o.t. \\
&=  - \sin x H( \pa_x D_x \om) + l.o.t. . 
\eal
\]
For $P_2$ and $P_4$, we have 
\[
\bal
\pa_x D_x P_2 & = (a-1) \pa_x ( (1-\cos x) D_x \om + \om D_x(1-\cos x) ) = (a-1) (1-\cos x) \pa_x D_x \om + l.o.t.. \\
\pa_x D_x P_4 & = (a-1) \pa_x ( \cos x D_x u + u D_x \cos x - u_x(0) D_x \sin x ) \\
& = (a-1) \pa_x ( \cos x \sin x  u_x -  u  \sin^2 x - u_x(0) \sin x \cos x ) \\
& = (a-1) \pa_x ( \cos x \sin x  (u_x -u_x(0)) )  + l.o.t. \\
& = (a-1) \pa_x( \cos x \sin x) \cdot (u_x - u_x(0)) + (a-1) \sin x \cos x \cdot u_{xx} + l.o.t. = l.o.t..
\eal
\]

Performing the weighted $H^1$ estimate on $D_x \om$ \eqref{eq:H2} with weight $\rho$ yields  
\beq\label{eq:H2_est1}
\bal
\f{1}{2} \f{d}{dt} \la \pa_x D_x \om, \pa_x D_x \om \rho \ra = \la \pa_x D_x \cL_a \om, \pa_x D_x \om \rho \ra + \la \pa_x D_x N(\om) + \pa_x D_x F( \bar \om ), \pa_x D_x \om  \rho \ra.
\eal
\eeq

Using the above estimates on $P_i$, we obtain 
\[
\bal
\la \pa_x D_x \cL_a \om, \pa_x D_x \om \rho \ra 
&\leq \B\la - a \sin x \pa_x \pa_x D_x \om  - \sin x H( \pa_x D_x \om) 
+ (a-1) (1-\cos x) \pa_x D_x \om  , \pa_x D_x \om \rho \B\ra  \\
&  + C || \pa_x D_x \om \rho^{1/2} ||_{L^2} || \om ||_{\cH}  \teq  \la I_1 + I_2 + I_3 , \pa_x D_x \om \ra + I_4,
\eal
\]
where we have bounded the l.o.t. in $P_i$ by $|| \om ||_{\cH}$. 

For $I_1$, applying Lemma \ref{lem:IBP} with $f = \pa_x D_x \om$ yields 
\[
\la I_1 ,  \pa_x D_x \om \rho \ra = -\f{a}{2} || \pa_x D_x \om \rho^{1/2}||^2_{L^2}
\leq ( -\f{1}{2} +\f{1}{2} |a-1|)  || \pa_x D_x \om \rho^{1/2}||^2_{L^2} .
\]

Note that $ \sin x \cdot \rho = \sin x (\sin \f{x}{2})^{-2} = 2 \cot \f{x}{2}$. For $I_2$, applying Lemma \ref{lem:tri} with $f = \pa_x D_x \om$, we obtain 
\[
\la I_2, \pa_x D_x \om \rho \ra =  -2 \la \cot \f{x}{2} H ( \pa_x D_x \om) , \pa_x D_x \om \ra
= 2 \pi  (H( \pa_x D_x \om)(0))^2.
\]
Since $ H f(0) = - \f{1}{2\pi} \la  \cot \f{x}{2}, f\ra$ and $\pa_x \cot \f{x}{2} = - \f{1}{2} (\sin \f{x}{2})^{-2}$, using integration by parts, we derive 
%and the Cauchy-Schwarz inequality, we derive 
\[
\bal
|H( \pa_x D_x \om)(0) | &= \f{1}{2\pi} |\la \cot \f{x}{2} ,\pa_x D_x \om \ra |
= C \B| \B\la  (\sin \f{x}{2})^{-2}, D_x \om \B\ra \B|  \\
& = C \B| \B\la  (\sin \f{x}{2})^{-2}, \sin x \pa_x \om \B\ra \B|  
\les || \om_x (\sin \f{x}{2})^{-1} ||_{L^2} \les || \om ||_{\cH}.
\eal
\]

Therefore, we establish 
\[
|\la I_2 ,  \pa_x D_x \om \rho \ra | \les || \om ||^2_{\cH}.
\]

The estimate of $I_3$ is simple 
\[
|\la I_3, \pa_x D_x \om \rho \ra| \les |a-1| \cdot || \pa_x D_x \om \rho^{1/2} ||^2_{L^2}.
\]

Therefore, we obtain the following estimates for $\cL_a$
\beq\label{eq:H2_est2}
\bal
\la \pa_x D_x \cL_a \om, \pa_x D_x \om \rho \ra 
&\leq (-\f{1}{2} + C|a-1|) || \pa_x D_x \om \rho^{1/2} ||^2_{L^2}
+ C || \pa_x D_x \om \rho^{1/2} ||_{L^2} || \om ||_{\cH} + C||\om||_{\cH}^2 \\
& \leq (-\f{3}{8} + C|a-1|) || \pa_x D_x \om \rho^{1/2} ||^2_{L^2}
+ C  || \om ||^2_{\cH} .\\
\eal
\eeq

\subsection{Estimates of the nonlinear and the error terms}\label{sec:NF}
Recall the nonlinear term and the error term $N(\om), F(\bar \om)$ \eqref{eq:NF}. Using \eqref{eq:profile},\eqref{eq:normal}, we have 
\beq\label{eq:NF2}
\bal
N(\om)  &= (a-1)u_x(0) \om + u_x \om  - a u \om_x  , \\ %\teq N_1 + N_2 + N_3, \\
 F(\bar \om) &= (a-1 + \cos x) (-\sin x) - a \sin x( -\cos x) \\
 &= (1-a) (\sin x - \sin x \cos x) = 2 (1-a) \sin x  \sin^2 \f{x}{2}.
\eal
\eeq

Since $F$ is of order $O(x^3)$ near $0$, it is easy to obtain 
\[
|| F(\bar \om) ||_{\cH} \les |1-a|, \quad  || \pa_x D_x F( \bar \om) \rho^{1/2} ||_{L^2} \les |1-a|,
\]
which implies
\beq\label{eq:F_est}
| \la F(\bar \om),  \om \ra_Y| \les |1-a| ||\om||_Y , \quad | \la \pa_x D_x F(\bar \om), \pa_x D_x \om \rho \ra| \les |1-a| ||\om||_X.
\eeq

To control the $L^{\infty}$ norm, we have a simple Lemma, whose proof is deferred to the appendix.
\begin{lem}\label{lem:linf}
Suppose that $ f_x \in L^2(\rho)$ and $f(0) = 0$. We have $ || f \rho^{1/2} ||_{L^{\infty}} \les 
||f_x \rho^{1/2} ||_{L^2}$.
\end{lem}

Using the Sobolev embedding and the above Lemma, we yield
\beq\label{eq:Linf1}
\bal
&|| u_x||_{L^{\infty}} \les || \om ||_{\cH}, \quad || \om \rho^{1/2} ||_{L^{\infty}} \les 
|| \om_x \rho^{1/2}||_{L^2} \les || \om ||_{\cH},  \\ 
& || D_x \om \rho^{1/2} ||_{L^{\infty}} \les || \pa_x D_x \om \rho^{1/2} ||_{L^2}. 
\eal
\eeq
In particular, $|u_x(0)| \les || \om ||_{\cH}$. With the above estimates and \eqref{eq:basic1}, the estimate of $\la N(\om), \om \ra_Y$ is standard. In the case of $a=1$, it has been established in \cite{lei2019constantin}. Its generalization to the case of $a$ close to $1$ is straightforward. In particular, we obtain
\beq\label{eq:H1_N}
|\la N(\om), \om \ra_Y| \les || \om ||^3_{\cH}.
\eeq

\paragraph{\bf{Weighted $H^2$ estimate of $N(\om)$ with odd assumption}}

Next, we estimate the nonlinear term $N(\om)$ in \eqref{eq:H2_est1}. We need to impose an extra condition on the initial perturbation $\om_0$. It is easy to see that the odd condition is preserved and both $u, \om$ are odd. Since the solution is $2 \pi $ periodic, we obtain $\om(\pi) = 0$ and $u(\pi) = 0$.

 Recall the $X$ norm in Definition \ref{def}. Using the estimate
\[
| \rho^{1/2} \pa_x D_x f - 2 \cos \f{x}{2} f_{xx} |
= | (\sin \f{x}{2} )^{-1} ( \cos x f_x + \sin x f_{xx} ) - 2 \cos \f{x}{2} f_{xx} | \les | f_x| \rho^{1/2} ,
\]
we derive the equivalence of norms 
\beq\label{eq:equiX}
 || f||_X^2 \les || f||^2_{\cH} + || \pa_x D_x f \rho^{1/2}||^2_{L^2}
 \les || f||_X^2
\eeq
for any $f \in X$. 

%To simplify the derivation, we use x.t. ($X$ terms) to denote the terms whose weighted $L^2(\rho)$ norm can be bounded by $C|| \om ||_X^2$ for some absolute constant $C$ and it can vary from line by line.

Recall the $L^{\infty}$ estimates \eqref{eq:Linf1}. We can control the $L^{\infty}$ norms of $u_x, \om \rho^{1/2}, D_x \om \rho^{1/2}$ by $|| \om ||_X$. Using \eqref{eq:L2_xt1}, the above equivalence and $\rho^{1/2} \sin x = \cos \f{x}{2}$, we have
\[
 || u_x||_{L^2} +|| u_{xx}||_{L^2} +  ||\rho^{1/2} \pa_x D_x \om||_{L^2} +  ||\rho^{1/2} \sin x \om_{xx} ||_{L^2} \les || \om||_X.
\]
% Hence, we yield 
% \beq\label{eq:xt1}
% \bal
%  & || u_x  \sin x \om_{xx}  \rho^{1/2} ||_{L^2} 
%   + || u_x \pa_x D_x \om \rho^{1/2} ||_{L^2}  + ||u_{xx} D_x \om \rho^{1/2} ||_{L^2} \\
% &+ || u_x \om_x \rho^{1/2} ||_{L^2} + || u \om_x \rho^{1/2}||_{L^2}
% \les || \om||_X^2.
% \eal
% \eeq

The crucial odd condition on the solution is used to obtain $u(0) = u(\pi)= 0$ and thus 
$|u | \les ||u_x||_{L^{\infty}} |\sin x|$. Hence, we obtain 
\beq\label{eq:xt2}
\bal
|| u \om_{xx} \rho^{1/2} ||_{L^2}  &= || \f{u}{\sin x} ||_{L^{\infty}} || \sin x \rho^{1/2} \sin x \om_{xx} ||_{L^2} \\
&\les || u_x ||_{L^{\infty}} || \rho^{1/2} \sin x \om_{xx}||_{L^2} \les || \om||^2_{X}.
% || u \om_{xx} \rho^{1/2} ||_{L^2}  = || \f{u}{\sin x}  \rho^{1/2} \sin x \om_{xx} ||_{L^2} \les || \f{u}{\sin x} ||_{L^{\infty}} || \cos \f{x}{2} \om_{xx}||_{L^2} \les || \om||^2_{X}.
\eal
\eeq

With the above preparations, the estimates of the $\la \pa_x D_x N(\om), \pa_x D_x \om \rho \ra$ are standard and similar to those in \cite{lei2019constantin}. We only focus on the difficult terms that require special estimates. The first term is $\pa_x D_x u_x \cdot \om$ that appears in the expansion of $ \pa_x D_x (u_x \om )$. Using Lemma \ref{lem:commu}, we get 
\[
\bal
  \pa_x D_x u_x \cdot \om  = \pa_x (  H(D_x \om) + \f{1}{2\pi} \la \om, \sin x\ra )  \cdot \om 
  = \pa_x  H(D_x \om)  \cdot \om =  \om  H(  \pa_x D_x\om ).
  \eal
\]

Using the $L^2$ isometry property of the Hilbert transform and \eqref{eq:Linf1}, we obtain 
\[
\bal
|| \om H( \pa_x D_x \om) \rho^{1/2}||_{L^2} &\les || H( \pa_x D_x \om) ||_{L^2} || \om \rho^{1/2}||_{L^{\infty}} \les || \pa_x D_x \om ||_{L^2} || \om||_X \les || \om||_X^2,
\eal
\]
which implies the $L^2(\rho)$ estimate of $  \pa_x D_x u_x \cdot \om $. The second term is $u \cos \om_{xx}$ that appears from $ -a \pa_x D_x ( u \om_x)$. Using $D_x \om_x = \pa_x( D_x \om) - \cos x \om_x$ and a direct calculation yields 
\[
\bal
\pa_x D_x (u \om_x) &= \pa_x ( \om_x D_x u + u D_x \om_x)
= \pa_x(  \om_x D_x u + u \pa_x D_x \om - u \cos x \om_x ) \\
& = \om_{xx} D_x u + \om_x \pa_x ( D_x u) 
+ u_x \pa_x D_x \om + u \pa_x ( \pa_x D_x \om) - \pa_x( u \cos x) \om_x - u \cos x \om_{xx} . \\
\eal
\]
The estimates of the first five terms are standard. For the last term, we apply \eqref{eq:xt2}. 

In summary, we obtain 
\beq\label{eq:H2_N}
\la \pa_x D_x N(\om), \pa_x D_x \om \rho \ra \leq  C || \om||_X^3.
\eeq

\subsection{Nonlinear stability}\label{sec:nonsta}
Recall the equivalence of norms in Lemma \ref{lem:linop}.
%Recall that the $\cH$ norm and $Y$ norm are equivalent in Lemma \ref{lem:linop}. 
Combining the estimates \eqref{eq:H1_est2}, \eqref{eq:H2_est2}, \eqref{eq:H1_N}, \eqref{eq:H2_N} and \eqref{eq:F_est}, we can construct an energy 
\[
E(t)^2  = || \pa_x D_x \om \rho^{1/2}||^2_{L^2} + \mu \la \om, \om \ra_Y
\]
for some absolute constant $\mu > 1$, such that the following estimate holds for \textit{odd} perturbation 
\[
\f{1}{2} \f{d}{dt} E(t) \leq ( -\f{1}{3}  + C|a-1| ) E(t)^2 
+ C ( || \om ||_X^3 + || \om||^3_{\cH}) + C |a-1| ( || \om ||_X + || \om ||_{\cH}).
\]

Since $\mu$ is absolute, using \eqref{eq:equiX} and the equivalence between the $\cH$ norm and the $Y$ norm from Lemma \ref{lem:linop}, 
%the fact that the $\cH$ norm and $Y$ norm are equivalent from Lemma \ref{lem:linop}, 
we yield the equivalence between the energy $E$ and the $X$ norm, i.e.
\beq\label{eq:equivEX}
 || \om||_X^2 \les E(t) \les || \om||_X^2.
\eeq
Therefore, we can further obtain 
\[
\f{1}{2} \f{d}{dt} E(t) \leq 
( -\f{1}{3}  + C|a-1| ) E(t)^2 + C E(t)^{3/2} + C|a-1| E(t)^{1/2}.
\]

Hence, there exists some small constant $\d_0$ and some absolute constant $c>0$ such that for $ |a-1| < \d_0$ , if $E(0)^{1/2} < c|a-1|$, then $E(t)^{1/2} < c|a-1|$ for all $t>0$, which can be proved by a bootstrap argument. Using this bootstrap result and \eqref{eq:equivEX}, we obtain 
\beq\label{eq:boots}
|| \om ||_X \leq C E(t)^{1/2} \leq C |a-1|.
\eeq
We can further choose smaller $\d_0$ such that 
\beq\label{eq:cw_est}
\bal
 |c_{\om}| & = | (1-a) u_x(0)| \leq C |a-1| E(t)^{1/2}  \\
 &\leq C|a-1|^2  \leq C \d_0 |a-1| < \f{1}{10} |a-1|.
 \eal
\eeq

\subsection{Convergence to the self-similar profile}

We focus on odd initial perturbation $\om_0 \in X$ with $E(0)^{1/2} \leq c|a-1|$. 
We obtain $\int \om_0 dx= 0$ and estimate \eqref{eq:boots} for $||\om||_X$.
%It automatically satisfies $\int \om_0 dx= 0$ and we have a-priori estimate $E(t)^{1/2} < c|a-1|$ for all $t> 0$. 
Moreover, since $\om$ is odd, we get $\om_e = \la \om , \cos x - 1 \ra_{\cH} = 0$ and similarly $\om_{e, t} = 0$. %As a result, %the $Y$ and the $\cH$ 
From Lemma \ref{lem:linop}, we obtain that the inner products $\la \cdot,\cdot \ra_Y,  \la \cdot, \cdot \ra_{\cH}$ are the same, and the norms $||\cdot ||_Y$ and $|| \cdot ||_{\cH}$ are the same.
%$Y$ norm and the $\cH$ norm are also the same. 

Performing estimate on $\la \om_t, \om_t \ra_Y$, we yield
\[
\f{1}{2} \f{d}{dt} \la \om_t , \om_t \ra_Y
= \la \cL_a \om_t , \om_t \ra_Y + \la \pa_t N(\om), \om_t \ra_Y \teq I + II.
\]
The estimate of the $\cL_a$ part follows from \eqref{eq:H1_est2}
\[
\la \cL_a \om_t , \om_t \ra_Y \leq (-\f{3}{8} + C|a-1|) || \om||_Y^2.
\]
A direct calculation yields 
\[
 \pa_x \pa_t( u\om_x) = \pa_x ( u_t \om_x + u \om_{t, x} )
 = u_{t,x} \om_x + u_t \om_{xx} + u_x \om_{t, x} + u \om_{t, xx}.
\]
We focus on the term $u_t \om_{xx}$.
%The estimate of $II$ is similar to that in Section \ref{sec:NF}. 
The inner product involving $u_t \om_{xx}$ is $-a \la u_t \om_{xx} , \pa_x \om_t \rho \ra.$
Using an estimate similar to \eqref{eq:xt2} and the $L^{\infty}$ estimate \eqref{eq:Linf1}, we get 
\[
|| u_t \om_{xx}\rho^{1/2} ||_{L^2} \les || \f{u_t}{\sin x}||_{L^{\infty}}|| \om||_X\les 
|| u_{t, x}||_{L^{\infty}} || \om||_X
\les || \om_t ||_{\cH} || \om||_X,
\]
where we have used $u_t(0) = u_t(\pi) =0$ due to the crucial odd condition to obtain the second inequality. Other terms in $II$ can be bounded by $|| \om_t||_{\cH} || \om||_{X}$ using estimates similar to those in Section \ref{sec:NF}. We can obtain 
\[
\f{1}{2} \f{d}{dt} || \om_t||_Y^2  \leq (- \f{3}{8} + C|1-a|  + C || \om||_X) || \om_t||_Y^2 \\
\leq (- \f{3}{8} + C|1-a|  ) || \om_t||_Y^2 .
\]
The last inequality holds due to \eqref{eq:boots}.
%the a-priori estimate $E(t)^{1/2} < c|a-1|$ and \eqref{eq:equivEX}. %$ || \om||_X < c|a-1|$. 

Now we choose $\d_1 = \d_0$ and $|a-1| < \d_1$, where $\d_0$ is the parameter determined in \eqref{eq:cw_est}.
%Section \ref{sec:nonsta}.
Applying the argument in \cite{chen2019finite}, we can obtain that $\om(t) + \bar \om$ converges to some $\om_a$ strongly in the $\cH$ norm and $c_{\om}+ \bar c_{\om} \to c_{\om, a}$ for some scalar $c_{\om, a}$ exponentially fast as $t\to \infty$. In addition, there is a subsequence of $\om(t) + \bar \om$ that converges weakly to $\om_a$ in $X$ and thus $\om_a \in X$. Moreover, $(\om_a,  c_{\om, a})$ is the steady state of \eqref{eq:DGdy}.  The relation \eqref{eq:relation} implies that $\om_a$ is a self-similar profile of \eqref{eq:DG}. Since $\om(t) + \bar \om$ is odd, the convergence $\om(t)  + \bar \om\to \om_a$ implies that $\om_a$ is odd. 
%The a-priori estimate $ ||\om(t)||_X < c|a-1|$ The estimate
Using the convergence results, $ ||\om(t)||_X \leq C|a-1|$ in \eqref{eq:boots} and \eqref{eq:cw_est}, we establish the estimates \eqref{eq:err_prof} in Theorem \ref{thm:profile}. 

\subsection{Regularity of the profile}
In this Section, we estimate the regularity of $\om_a$.

\subsubsection{Estimate of exponent $\g(a)$}  Recall $\bar \om = -\sin x$ and $H\bar \om(\pi) = \bar u_x(\pi) = -1$ in \eqref{eq:profile}.  Using \eqref{eq:err_prof}, we yield
\[
\bal
&|c_{\om, a} - (a-1)| \les |a-1|^2, \quad | c_{\om, a} | \les |a-1|,  \\
&| H\om_a(\pi) - H \bar \om(\pi)| \les ||\om - \bar \om||_X \les |a-1| , 
\quad |H\om_a(\pi)| \les 1.
\eal
\]

It follows $| H\om_a(\pi)| \geq |H \bar \om(\pi)| - C|a-1| = 1 - C|a-1|$. By choosing smaller $\d_1$ so that $|a-1| < \d_1$ is small, we obtain $ | H\om_a(\pi)|  \geq \f{1}{2}$. Applying the above estimates, we obtain 
\[
 \B| \f{  c_{\om,a} + (1-a) H \om_a( \pi) }{ a H\om_a(\pi)} -  \f{ (a-1) + (1-a) H \bar \om (\pi)}{ a H \bar\om(\pi) }\B| \les |a-1|^2.
 \]

The first term on the left hand side gives $\al(a)$ in \eqref{eq:hol_exp} and the second term equals $ \f{(a-1) +(a-1)}{-a}$ which satisfies
\[ | \f{(a-1) +(a-1)}{-a} - 2(1-a)| \les |1-a|^2. \]

By further choosing smaller $\d_1$ so that $|a-1| < \d_1$ is small, we can obtain $ | H\om_a(\pi) + 1| = |H\om_a(\pi) - H \bar \om(\pi)| < \f{1}{10}$ and $|\al(a)- 2(1-a)| \leq \f{1}{2}|a-1|$. We complete the estimates of $\al(a)$ and $H\om_a(\pi)$ in Theorem \ref{thm:profile}.

% Recall $\bar c_{\om}, \bar u$ in \eqref{eq:profile}. Since $c_{\om, a}$ is close to $\bar c_{\om} = a-1$ and $H\om_a(\pi)$ is close to $\bar u_x(\pi) = -1$, the H\'older exponent in Theorem \ref{thm:profile}
% \beq\label{eq:hol_exp}
% \al(a)  =  \f{  c_{\om,a} + (1-a) H \om_a( \pi) }{ a H\om_a(\pi)}
% \eeq
% is close to $\f{ 2(a-1)}{ -1} = 2(1-a)$. The rigorous estimate of $\al(a)$ is similar to that in \eqref{eq:cw_est} and we can further choose smaller $\d_1$ such that 
% \[
% |\al(a) - 2(1-a)| < \f{1}{2}|1-a|, \quad |H\om_a(\pi) + 1| < \f{1}{2}
% \]
% This completes the estimate of $\al(a)$.

\subsubsection{Estimates of the H\"older norm of $\om_a$} 
%In this Section, we estimate the regularity of $\om_a$ stated in Theorem \ref{thm:profile}.
%prove the statement in Theorem \ref{thm:profile} on the regularity of $\om_a$. 
Applying $|| \om_a + \sin x||_X \les |a-1|$ in \eqref{eq:err_prof}, we obtain 
\[
|| u_{a,xx} + \sin x||_{L^2} \les |a-1|, \quad || \om_{a, xx} - \sin x ||_{L^2 ( [-\pi /2, \pi /2] ) } \les |a-1| .
\]

Using Sobolev embedding , we yield that $\om$ is close to $-\sin x$ on $[-\pi /4, \pi /4]$ in $C^{1,1/3}$ norm and $u$ is close to $ \sin x$ in $C^{1,1/3}$ norm. Hence, up to further choosing smaller $\d_1$ so that $|a-1| <\d_1$ is small, we can assume that $ u_a$ has exactly two zeros at $x=0, x =\pi$ and $\om_a$ has only one zero at $x=0$ in $[-\pi /4, \pi /4]$. We simplify $\om_a,  c_{\om, a}$ as $\om, c_{\om}$ in the following derivation. Since $(\om, c_{\om})$ is the steady state of \eqref{eq:DGdy}, we derive 
\beq\label{eq:ODE0}
a u\om_x = (c_{\om} + u_x) \om .
\eeq

We fix $ x_0 = \f{\pi}{8}$. Then $\om( x_0) \neq 0$. The above equation can be seen as an ODE on $\om$ with given $ c_{\om}, u, u_x$. Starting from $x_0$, we can solve the ODE 
\beq\label{eq:ODE}
\om(x) = \om(x_0) \exp( \int_{x_0}^x \f{ c_{\om} + u_x }{a u} dx ).
\eeq
%Existence and uniqueness theorems on ODE shows that we can solve the ODE in $(x_-, x_+)$ as long as $\om$ is bounded. Since we know $\om \in H^1 \hookrightarrow C^{1/3}$, the above relation holds for all $x \in S^1$. 

Since $ u \in C^{1, 1/3}$ and $u(\pi) =0$, we yield
\[
 |u(x)  -u_x(\pi)(x-\pi)| \les_a |x-\pi|^{4/3}, \quad |u_x(x) - u_x(\pi)| \les_a |x-\pi|^{1/3} .\]
It follows 
\beq\label{eq:asym1}
\B|\f{ c_{\om} + u_x(x)}{a u(x)} - \f{c_{\om} + u_x(\pi)}{ a u_x(\pi) (x-\pi)}\B|
\les_a |x-\pi|^{-2/3}.
\eeq

Recall the formula of $\al(a)$ in \eqref{eq:hol_exp}. We get 
\[
\B|\int^x_{x_0} \f{ c_{\om} + u_x(x)}{a u(x)}  - ( \al(a) + 1) \f{1}{x-\pi} dx \B| 
\les_a  \int_{x_0}^x |x-\pi|^{-2/3} dx  \les_a 1
\]
for all $x \in ( \pi/2, \pi)$. Note that 
$ \int_{x_0}^x  \f{1}{ x-\pi} dx = \log  |x-\pi| - \log( |x_0 - \pi| $. 
Plugging these estimates into \eqref{eq:ODE}, for $x \in (\pi/2, \pi) $, we obtain
\beq\label{eq:asym2}
 |x-\pi|^{\al(a) + 1} \les_a |\om(x) | \les_a | x-\pi|^{ \al(a)+1}.
\eeq

From the estimates of $\al(a)$ following \eqref{eq:hol_exp}, we have $\al(a) < 0$ for $a >1$ and $\al(a) > 0$ for $a < 1$. Therefore, for $ 1 < a< 1+ \d_1$ and any $1 + \al(a) < a < 1$, we obtain 
\[
\liminf_{x \to \pi^-} \f{ |\om(x) |}{ |x - \pi|^{\al} }  \gtrsim_a  \liminf_{x\to \pi^-}  |x-\pi|^{\al(a) + 1 -\al} = \infty.
\]
Since $\om(\pi) = 0$, we yield $\om_a \notin C^{\al}$ for $1-\d_1<a< 1$ and $\al(a) + 1<\al <1$.

Recall $u \in C^{1,1/3}, \om \in H^1(S^1) \hookrightarrow C^{1/3}(S^1)$ and $u(x) \neq 0$ for $ x \neq 0, \pi$. Using \eqref{eq:ODE0}, we can derive that $\om_x$ is continuous on $(0,\pi)$ and thus $\om \in C^1( (0,\pi))$.  Similarly, we can obtain $\om \in C^1( (-\pi, 0))$. From the discussion at the beginning of this section, we have $\om \in C^{1,1/3}( -\pi/4, \pi/4)$. It follows $ \om \in C^1( S^1 \bsh \{\pi\})$. It remains to show that $\om_x$ is continuous at $x = \pi$. Using \eqref{eq:ODE0}, \eqref{eq:asym1}, \eqref{eq:asym2} and $\al(a) > 0$ for $1- \d_1 < a <1$, we obtain 
\[
 \limsup_{x \to \pi^-} |\om_x(x)| =  \limsup_{x\to \pi^-} | \f{c_{\om} + u_x}{ au}\om|
\les_a   \limsup_{x\to \pi^-} \f{1}{|x-\pi|} \cdot |x-\pi|^{1 + \al(a)} =0,
\]
Similarly, we yield $  \limsup_{x \to \pi^+} |\om_x(x)|=0 $. Thus $\om \in C^1(S^1)$. In particular, $\om_x(\pi) = 0$.

Finally, for any $\al$ with  $\al(a) < \al< 1$, using a similar argument and \eqref{eq:asym2} yields 
\[
 \liminf_{x \to \pi^-} \f{ |\om_x(x)|}{ |x-\pi|^{\al} }  =  \liminf_{x\to \pi^-} \B| \f{c_{\om} + u_x}{ au}\B| \f{ |\om| }{ |x-\pi|^{\al}}
\gtrsim_a  \liminf_{x\to \pi^-}  \f{1}{|x-\pi|} \cdot |x-\pi|^{1 + \al(a) -\al} = \inf.
\]
We conclude that $\om_a \notin C^{1, \al}$ for any $ \al(a) <\al < 1$.

So far, we conclude the proof of Theorem \ref{thm:profile}.

\section{Stability of the self-similar profiles}\label{sec:stability}

In this section, we establish the stability of the profiles $(\om_a,  c_{\om,a}), |a-1|<\d_1$ constructed in the previous section and prove Theorem \ref{thm:stability}.

\subsection{Stability analysis}

Denote by $u_{\al}, u_{a,x}$ the velocity field corresponding to $\om_a$. We linearize \eqref{eq:DGdy} around the steady state $(\om_a, c_{\om,a})$ and consider odd perturbation $\om_0 \in \cH$
%Instead of linearizing \eqref{eq:DGdy} around the approximate steady state in \eqref{eq:profile}, we linearize it around the exact steady state $(\om_a, c_{\om,a})$ and consider odd perturbation $\om_0 \in \cH$
\beq\label{eq:lin_ex}
\om_t  = - a u_{\al} \om_x + ( c_{\om, a} +  u_{a,x}) \om + (c_{\om} + u_x ) \om_a - a u \om_{a,x} + N( \om )   \teq \cT_a \om +   N(\om ) ,
\eeq
where $\cT_a$ is the linearized operator, $c_{\om}$ and the nonlinear term $N(\om) $ are the same as that in \eqref{eq:normal} and \eqref{eq:NF}. Here, we do not have the error term since $(\om_a, c_{\om,a})$ is the steady state. Since the odd condition on $\om$ is preserved, we have $\om(\pi, t) = u(\pi, t) = 0$ for $t >0$.

 We compare $\cT_a$ with $\cL_a$ in \eqref{eq:lin} or \eqref{eq:linop}.
\[
\bal
\cT_a \om  - \cL_a  \om =& -a ( u_a - \bar u ) \om_x + ( c_{\om, a} -\bar c_{\om} + u_{\al, x} - \bar u_x) \om \\
& + (c_{\om} + u_x) ( \om_{a} - \bar{\om}) - a u (\om_{a,x} - \bar \om_x ) \teq \cR_a \om,
\eal
\]
where $\bar \om = -\sin x, \bar u = \sin x, \bar c_{\om} = a-1$ are given in \eqref{eq:profile}. We focus on the last term in $\cR_a \om$
%Since $\om_e = 0$, we have 
\[
\bal
 (u(\om_{a,x} - \bar \om_x) )_x =  u_x (\om_{a,x} - \bar \om_x) + u (\om_{a,xx} - \bar \om_{xx} )\teq I + II.
\eal
\]
For $II$, we use the crucial condition $u(\pi,t)=0$ due to the odd condition and an estimate similar to \eqref{eq:xt2} to obtain 
\[
|| II \rho^{1/2}||_{L^2} \les || \f{u}{\sin x}||_{L^{\infty}} || (\om_{a,xx} - \bar{\om}_{xx}) \sin x\rho^{1/2} ||_{L^2} \les || u_x||_{L^{\infty}} || \om_a - \bar{\om} ||_X \les |a-1| || \om||_{\cH}.
\]
The last inequality is due to  \eqref{eq:err_prof} and \eqref{eq:Linf1}. 

Recall the definitions of $\om_e$, the $Y$ norm and inner product in Lemma \ref{lem:linop}. Since $\om$ is odd,  $\om_e = 0$ vanishes and we get $|| \om ||_{\cH} = || \om ||_Y$. 
With the control \eqref{eq:err_prof} on the error $|| \om_a - \bar{\om}||_X$ and $| c_{\om,a} - \bar c_{\om}|$, the $I$ term and other terms in $\cR_a \om$ can be estimated in a way similar to that in Section \ref{sec:H1}.  In particular, we can obtain 
\[
|\la \cR_a \om, \om \ra_Y | =| \la \pa_x ( \cR_a \om), \om_x \rho \ra | 
\les |a-1| || \om||^2_{\cH},
\]
which along with \eqref{eq:H1_est2} implies 
\[
\la \cT_a \om, \om\ra_Y \leq ( -\f{3}{8} + C|a-1|) || \om||_Y^2.
\]

The estimate of $N(\om)$ is essentially the same as that in Section \ref{sec:NF} and we can establish estimates similar to \eqref{eq:H1_N}. In summary, we yield 
\[
\f{1}{2} \f{d}{dt} || \om||_{\cH}^2 \leq ( -\f{3}{8} + C_1|a-1|) || \om||_{\cH}^2 + C_1 || \om ||_{\cH}^3
\]
for some absolute constant $C_1 >0$, where we have replaced the $Y$ norm by the $\cH$ norm since they are the same. Therefore, there exist small positive parameters $\d_{20} < \d_1 $ and $\d_{30}$, such that for $|a-1| < \d_{20}$, if $|| \om_0||_{\cH} < \d_{30}$ then  
% Now, we can choose $0 <\d_{20} < \d_1$ and $\d_{30} > 0$ such that $ C_1 \d_{20} + C_1 \d_{30} -\f{3}{8}< -\f{1}{3}$. Applying a bootstrap argument, we obtain that if $|| \om_0||_{\cH} < \d_{30}$ then 
\beq\label{eq:conv}
\f{1}{2} \f{d}{dt} || \om||^2_{\cH} < -\f{1}{3}  || \om||^2_{\cH}.
\eeq
As a result, $|| \om||_{\cH}$ decays exponentially fast. We can further choose smaller $\d_{20}, \d_{30}$ such that 
\[
\bal
 |u_x(0) + u_{a,x}(0) - 1| &\leq |u_x(0)| + |u_{a,x}(0) - 1|
 \leq C || \om||_{\cH} + C || \om_a + \sin x ||_X \\
 &\leq C || \om_0||_{\cH} + C |a-1| \leq C (\d_{20} + \d_{30}) < \f{1}{10}.
\eal
\]
This ensures that $c_{\om,a} + c_{\om} = (a-1)( u_x(0) + u_{a,x}(0) )$ satisfies 
\beq\label{eq:cw}
sign( c_{\om,a} + c_{\om} ) = sign(a-1), \quad \f{9}{10} |a-1| \leq |c_{\om,a} + c_{\om} |  \leq \f{11}{10} |a-1|.
\eeq

Note that $\om_0 + \om_a$ is the initial data of \eqref{eq:DGdy}. Suppose that $|| \om_0 + \om_a + \sin x||_{\cH} < \d_3$ and $|a-1| < \d_2$ for $\d_2, \d_3$ to be determined. Using triangle inequality, we obtain 
\beq\label{eq:triangle}
|| \om_0||_{\cH} \leq  || \om_0 + \om_a + \sin x||_{\cH} + || \om_a + \sin x ||_{\cH}
 < \d_3+ C|a-1|  \leq \d_3 + C \d_2.
\eeq

We choose $\d_2, \d_3 >0$ in Theorem \ref{thm:stability} such that  
\beq\label{eq:para_thm2}
0 < \d_2 < \min(\d_{20}, \f{1}{100}), \quad \d_3 + C\d_2 < \d_{30} .
\eeq

As a result, if the initial data $\om_0 + \om_a$ of \eqref{eq:DGdy} satisfies $||\om_0 + \om_a + \sin x||_{\cH} < \d_3$, then $|| \om_0 ||_{\cH} < \d_{30}$ and thus we obtain the estimates \eqref{eq:conv}, \eqref{eq:cw}.

\subsection{ Rescaling}
Recall the rescaling relations \eqref{eq:rescal1}-\eqref{eq:rescal3}. To avoid confusion, we use $\tau, \om^{DR}$ to represent the temporal variable and solution in \eqref{eq:rescal2} and $t, \om^{phy}$ to represent those in the physical equation \eqref{eq:DG}. Denote by $c(\tau) = c_{\om}(\tau) + c_{\om,a}$ the scaling factor. By definition, we have $\om^{DR} = \om + \om_a$. Here, $c_{\om}(\tau), \om(\tau)$ are the perturbations and $\om^{DR}, c(\tau)$ solves \eqref{eq:DGdy}.

 The relations \eqref{eq:rescal1}-\eqref{eq:rescal3} imply
\beq\label{eq:rescal4}
\om^{phy}( x, t(\tau)) = C^{-1}_{\om}(\tau) \om^{DR}( x, \tau), \quad C_{\om}(\tau) = C_{\om}(0) \exp( \int_0^{\tau} c(s) ds)
\eeq
and $t(\tau) = \int_0^{\tau} C_{\om}(s) ds.$

Suppose that $\om_0^{phy}$ is odd and $|| \om^{phy}_0 + \sin x||_{\cH} < \d_3$. We choose $C_{\om}(0)=1$ so that $\om_0^{DR} = \om_0^{phy}$. Estimate \eqref{eq:triangle} and its following discussion implies \eqref{eq:conv} and \eqref{eq:cw}. Therefore, for $1 < a< 1+ \d_2$, we obtain $c(\tau) > \f{9}{10}|a-1|$, while for $ 1-\d_2 < a < 1$, we get $-\f{11}{10} |a-1| < c(\tau) < -\f{9}{10}|a-1|$. The discussion in Section \ref{sec:dsform} implies the blowup result for $1-\d_2 < a< 1$ and the long time behavior of the solution for $1<a< 1+ \d_2$ in Theorem \ref{thm:stability}. It remains to establish the estimates in Theorem \ref{thm:stability}.

\subsubsection{ Estimate of $\lam(t)$}\label{sec:lamt}
We choose $\lam( t(\tau)) = C_{\om}(\tau)$. 

For $ 1-\d_2 < a< 1$, since $c(s) <0$, $\lam(t(\tau))$ is decreasing. Using $c(s) = c_{\om} + c_{\om, a}$, \eqref{eq:cw} and the formula \eqref{eq:rescal4}, we obtain the estimate of the blowup time 
\[
T = t(\infty) = \int_0^{\infty} \exp(\int_0^{\tau} c(s) ds) d \tau
\geq \int_0^{\infty} \exp( - \f{11}{10} |a-1|\tau ) d \tau  \geq \f{1}{2|a-1|}
\]
and the estimate of $\lam(t(\tau))$
%Using $c(s) <0$ and \eqref{eq:rescal4}, we yield that $\lam(t(\tau))$ is decreasing and it satisfies 
\beq\label{eq:lamt}
\f{T - t(\tau)}{\lam( t(\tau))} = \f{ \int_{\tau}^{\infty} C_{\om}(s) ds}{ C_{\om}(\tau)}
= \int_{\tau}^{\infty} \exp( \int_{\tau}^s c(z) dz) d s
= \int_0^{\infty} \exp( \int_0^s c( \tau + z) dz ) ds ,
\eeq
where we have used a change of variable $ s \to \tau + s$ in the last equality. Since the perturbation $c_{\om}(\tau), \om(\tau)$ decays exponentially fast in $\tau$ (see \eqref{eq:conv}), we have  $c(\tau) = c_{\om}(\tau) + c_{\om,a} \to c_{\om, a} < 0$. Note that $\exp( \int_0^s c( \tau + z) dz ) \leq \exp( -\f{9}{10}|a-1| s) $ is integrable. Applying Dominated Convergence Theorem yields 
\[
\f{T -t(\tau)}{ \lam(t(\tau))} \to \int_0^{\infty} \exp( c_{\om, a} s) ds = - \f{1}{c_{\om, a}}.
\]
Taking the inverse of the above limit implies $\f{ \lam(t(\tau)}{T - t(\tau)} \to -c_{\om, a}$.

Similarly, for $ 1 < a < 1 + \d_2$, we can obtain $\lam(t(\tau))$ is increasing and $\f{ \lam(t(\tau)}{ t(\tau)} \to c_{\om, a}$.

\subsubsection{Convergence estimates}
Next, we establish the convergence estimate. Using \eqref{eq:rescal4}, $\lam(t(\tau)) = C_{\om}(\tau)$ and $\om^{DR} = \om + \om_a$, we have 
\[
 || \om^{phy} - \lam(t(\tau))^{-1} \om_a ||_{\cH} = C_{\om}(\tau)^{-1} || \om^{DR} - \om_a ||_{\cH}
 =  C_{\om}(\tau)^{-1} || \om ||_{\cH}.
\]
Using \eqref{eq:conv} and $ \om_0^{phy}= \om_0^{DR} = \om_0 + \om_a$, we further obtain 
\[
 || \om^{phy} - \lam(t(\tau)^{-1} \om_a ||_{\cH}  \leq C_{\om}(\tau)^{-1} e^{ -\f{\tau}{3}} || \om_0||_{\cH} 
 = C_{\om}(\tau)^{-1} e^{ -\f{\tau}{3}} || \om_0^{phy} - \om_a||_{\cH}.
\]
For $a \neq 1$, applying \eqref{eq:cw} to $c(s) = c_{\om} + c_{\om, a}$ and using \eqref{eq:rescal4}, we obtain 
\[
(\max( C_{\om}^{-1}(\tau) , C_{\om}(\tau) )  )^{ \f{1}{4( |1-a|)} + 1}
\leq \exp\B( \f{11}{10} |a-1| \tau \cdot (\f{1}{4|1-a|} + 1)  \B)\leq \exp( \f{\tau}{3}),
\]
where we have used $|a-1| < \d_2 < \f{1}{100}$ from \eqref{eq:para_thm2} in the last inequality. Combining the above two estimates and substituting $\lam(t(\tau)) = C_{\om}(\tau)$, we prove 
\beq\label{eq:conv2}
 || \om^{phy} - \lam(t(\tau))^{-1} \om_a ||_{\cH} \leq (\max( \lam(t(\tau)), \lam(t(\tau))^{-1} ))^{- \f{1}{ 4|a-1|}} || \om_0^{phy} - \om_a||_{\cH}.
\eeq
Since $\lam(t(\tau)) \leq 1$ for $ 1-\d_2 < a < 1$ and $\lam(t(\tau)) \geq 1$ for $1 < a< 1 + \d_2$, it follows the convergence estimates in Theorem \ref{thm:stability}.

\subsubsection{Growth of $|| \om^{phy}_x||_{L^{\infty}}$}

We focus on $ 1 < a < 1 + \d_2$ and further assume that $\om^{phy} \in H^s, s > \f{3}{2}$. The decay estimates \eqref{eq:conv2} and the BKM-type blowup criterion implies that $\om$ remains in $H^s$. Since $\om^{phy}$ is odd, we have $u^{phy}(\pi) = \om^{phy}(\pi) = 0$. The evolution of $\om^{phy}_x(\pi)$ is given by 
\[
\pa_{t} \om^{phy}_x(\pi) =  (1-a) u^{phy}_x(\pi) \om^{phy}_x(\pi).
\]
To avoid using $\om_{xx}$, the above ODE can be established by 
%either applying the flow map $ \pa_t \Phi(x, t) = au(  \Phi(x,t), t), \Phi(x , 0) = x$ in \eqref{eq:DG} or 
dividing both sides of \eqref{eq:DG} by $(x-\pi)$ and then taking $x\to \pi$ or using the flow map. Solving the ODE, we obtain 
\[
\om_x^{phy}( \pi, t ) = \exp\B( (1-a) \int_0^{ t } u_x^{phy}(\pi, s ) ds \B) \om^{phy}_{0,x}( \pi).
\]

Using the fast convergence \eqref{eq:conv2}, for any $t >0$, we obtain 
\[
S(t) \teq (1-a)\int_0^t u_x^{phy}(\pi, s ) ds  
\geq (1-a) \int_0^t \f{1}{\lam(s)} u_{a,x}(\pi) ds - C(a)  \geq - C(a).
\]

%From \eqref{eq:err_prof}, we have $ |u_{a,x}+1| \les |a-1|, | \om_{a,}|\$
From the estimate following \eqref{eq:hol_exp}, we have $u_{a, x}(\pi) < 0$. Hence $(1-a) u_{a,x}(\pi) > 0$. From the estimate of $\lam(t)$ in Section \ref{sec:lamt}, 
%From the argument in proving $ \f{\lam(\tau)}{T - t(\tau)} \to -c_{\om,a}$, %for $a<1$ and its natural generalization to prove $\f{ \lam(\tau)}{ t(\tau)} \to c_{\om,a}$ for $a>1$, 
it is not difficult to obtain that for any $\e > 0$, there exists $ C(\e, a)$, such that $\f{ \lam( t(\tau))}{ t(\tau)} < c_{\om, a} +\e $ for $t(\tau) > C(\e, a)$ (if $t(\tau)$ large, then $\tau$ must be large). Then for $t>1$ and $0<\e$ we get
\[
S(t) \geq (1-a) u_{a,x}(\pi) \int_1^{t}  \f{ 1}{ (c_{\om, a} + \e ) s } ds - C( a, \e)
= \f{(1-a) u_{a,x}(\pi) }{ c_{\om, a} + \e} \log(t) - C(a, \e).
\]
Combining the above estimates, we prove 
\[
| \om_x^{phy}(\pi, t)| \gtrsim_{a, \e} t^{ \g} |\om_{0, x}(\pi)|
\]
with $\g = \f{(1-a) u_{a, x}(\pi)}{ c_{\om, a } + \e}$. Since $\e >0$ is arbitrary, we establish the growth of $|\om_x(\pi, t)|$ in Theorem \ref{thm:stability}. The estimate $|\g(a)-1| \les |a-1|$ follows from \eqref{eq:err_prof}.

So far, we conclude the proof of Theorem \ref{thm:stability}.

\section{Proof of Theorem \ref{thm:weak}}\label{sec:weak}

The proof is similar to that of Theorem \ref{thm:stability}. We consider the perturbation around the approximate steady state $\om_0 = -\sin x, \bar c_{\om} = a-1$ \eqref{eq:profile}. In Sections \ref{sec:H1}, \ref{sec:NF}, we have obtained the following estimates for the perturbation $\om(t)$ with $\om_0 \in \cH$ and $\int_{S^1} \om_0 dx =0$ under normalization condition \eqref{eq:normal} on $c_{\om}$
\[
\f{1}{2} \f{d}{dt} || \om||_Y^2 \leq (-\f{3}{8} + C_2|a-1|) || \om||_Y^2 
+ C_2 ||\om||_Y^3 + C_2|a-1| \cdot || \om||_Y,
\]
where we have used the equivalence between the $Y$ norm and the $\cH$ norm in Lemma \ref{lem:linop}. Remark that we do not require that $\om_0$ is odd to obtain the weighted $H^1$ estimates of linear, nonlinear and the error terms. It follows that there exist $\d_4, \d_5 > 0$ 
%with 
% \[
% - \f{3}{8} + C_2 \d_4 < -\f{1}{4}, \quad  -\f{1}{4} \d_5 + C_2 \d_4  + C_2 \d_5^2 < 0,
% \]
such that for $|a-1| < \d_4$ and any $||\om_0||_Y < \d_5$, the bootstrap assumption $|| \om(t)||_Y <\d_5$ holds for any $t> 0$. 
Using the equivalence of norms in Lemma \ref{lem:linop} again and this bootstrap result, we obtain that if $||\om_0||_{\cH} < \d_5$, then $||\om_0||_Y < \d_5$, which further implies $|| \om(t)||_Y < \d_5$ and $|| \om(t)||_{\cH} \leq 2 || \om(t)||_Y < 2 \d_5$. The factor $2$ in the upper bound $2 \lam(t)^{-1} \d_5$ in Theorem \ref{thm:weak} is due to this equivalence. 

Up to further choosing smaller $\d_4, \d_5$, using the bootstrap result, we can obtain $|c(s) - (a-1)| \leq \f{|1-a|}{10}$ similar to \eqref{eq:cw}, where $c(s) = c_{\om} + a-1$. Plugging this estimate into \eqref{eq:lamt} yields the estimate of $ \f{\lam(t)}{T-t}$ for $1-\d_4 < a < 1$. $\f{\lam(t)}{t}$ is estimated similarly. Using the bootstrap result and the argument in Section \ref{sec:stability}, we can prove other results in Theorem \ref{thm:weak}. We omit the details.

%The estimates of $ \f{\lam(t)}{T-t}$ and $\f{\lam(t)}{t}$ are obtained by modifying those in Section \ref{sec:lamt}. Instead of using $c(s) \to c_{\om,a}$ exponentially fast in $s$, we apply $|c(s) - (a-1)| \leq \f{|1-a|}{10}$ for all $s>0$ to obtain the upper and lower bound on $ \f{\lam(t)}{T-t}$ and $\f{\lam(t)}{t}$. The estimate on $c(s) = (a-1) (1 + u_x(0))$ can be obtained from the bootstrap result and the smallness of $\d_4, \d_5$ (we can further choose smaller values if neccessary). 

\section{ An approach to obtain potential finite time blowup for other $a<1$}\label{sec:approach}

%The analysis presented in this paper benefits from the smallness of $|a-1|$, which enables us to construct the explicit approximate steady state \eqref{eq:profile} and use the analysis of the operator $\cL_1$ developed in \cite{lei2019constantin}. 
We discuss an approach which has the potential to be applied to obtain finite time blowup of \eqref{eq:DG} on a circle for other $a<1$ from smooth initial data. It is based on the method in \cite{chen2019finite}.

 %, we suggest the following approach.
\vspace{0.1in}
\paragraph{\bf{Construction of approximate steady state} }
The first step is to construct the approximate steady state using the dynamic rescaling equation 
\[
\om_t + a u\om_x = (c_{\om} + u_x) \om,
\]
which is the same as \eqref{eq:DGdy}, with normalization condition $c_{\om} = (a-1 ) u_x(0)$. 

For $a$ close to $1$, e.g. $0.95 < a< 1$, $\om_0 = -\sin x$ provides a good candidate for initial data. An approximate steady state $(\bar \om, \bar c_{\om})$ can be obtained by solving the above dynamic rescaling equation for long enough time numerically. The approximation error can be estimated a posteriori. If the error is sufficiently small, we can further perform stability analysis around $(\bar \om, \bar c_{\om})$. See more discussions in \cite{chen2019finite}. For $a$ away from $1$, e.g. $a <0.95$, the initial data can be chosen successively based on the approximate steady state for larger $a$, if it exists. Our preliminary numerical results suggest that the solution converges to some profile and the approximation error $ F(\om) \teq (c_{\om} + u_x) \om - a u\om_x$ decays rapidly in time. For $a$ away from $1$, the approximate steady state can also be constructed using the method in \cite{lushnikov2020collapse}.

\vspace{0.1in}
\paragraph{\bf{Stability analysis}}
Once an approximate steady state is constructed, one can follow the steps in Sections \ref{sec:H1} and \ref{sec:NF}. The key step is to establish the linear stability. For $a$ close to $1$, e.g. $0.95 < a< 1$, it is conceivable that linear stability can be established in a way similar to that in Section \ref{sec:H1} by applying the analysis of $\cL_1$ in Lemma \ref{lem:linop} (established in \cite{lei2019constantin}) and controlling the difference between $\cL_1$ and the new linearized operator. For $a$ away from $1$, linear stability may be established by an energy estimate similar to that in \cite{chen2019finite} using some well-chosen singular weight.

\appendix

\section{}
\begin{lem} \label{lem:tri}
Suppose that $f \in L^2( \sin^{-2}\f{x}{2})$. We have 
\beq\label{eq:tri}
\int_{S^1} \cot \f{x}{2} f \cdot Hf dx =  - \pi H f(0)^2.
\eeq

\end{lem}

\begin{proof}
Firstly, we consider $f \in C^{\infty}$. Using the Tricomi identity of the Hilbert transform (see e.g. \cite{Elg17,chen2019finite}), we obtain 
\[
\int_{S^1} \cot \f{x}{2} f \cdot Hf dx = - 2 \pi H( f \cdot Hf)(0)  = - \pi ( (Hf(0))^2 - f(0)^2).
\]
Since $f \in L^2( \sin^{-2} \f{x}{2})$, we have $f(0) = 0$ and obtain \eqref{eq:tri} for $f \in C^{\infty}$. For general $f$, we can find a sequence $f_n \in C^{\infty}$ such that $f_n \to f$ in $L^2( \sin^{-2} \f{x}{2})$. Clearly, we have $H f_n \to Hf $ and $f_n \cot \f{x}{2} \to f \cot \f{x}{2}$ in $L^2$. Using the Cauchy-Schwarz inequality, we get $ Hf_n(0) \to Hf(0)$. Applying \eqref{eq:tri} to $f_n$ and then taking $n\to \infty$ concludes the proof.
\end{proof}

Next, we prove Lemmas \ref{lem:commu} and \ref{lem:linf}.
\begin{proof}[Proof of Lemma \ref{lem:commu}]
Applying integration by parts yields 
\[
\bal
D_x Hf(x) &=\f{1}{2\pi} \int \sin x  f(y) \cdot  \pa_x \cot \f{x-y}{2} dy 
= -\f{1}{2\pi} \int \sin x f(y) \pa_y  \cot \f{x-y}{2} dy  \\
& = \f{1}{2\pi} \int  \sin x f_y(y) \cot \f{x-y}{2} dy . \\
\eal
\]
It follows 
\[
\bal
D_x H f(x) - H(D_x f )(x)  &= \f{1}{2\pi} \int (\sin x -\sin y) f_y(y)   \cot \f{x-y}{2} dy \\
%& = \f{1}{2\pi} \int 2 \cos \f{x+y}{2}2 \sin \f{x-y}{2} \cos \f{x+y}{2} f_y(y)  dy
\eal
\]
Note that $(\sin x - \sin y) \cot \f{x-y}{2} = 2 \sin \f{x-y}{2} \cos \f{x+y}{2} \cot \f{x-y}{2} = 2 \cos \f{x+y}{2} \cos \f{x-y}{2} = \cos x + \cos y.$ We conclude
\[
D_x H f(x) - H(D_x f )(x)  = \f{1}{2\pi} \int (\cos x + \cos y) f_y(y) dy =
 \f{1}{2\pi} \int \sin y \om(y) dy.
\]
\end{proof}

%Next, we prove Lemma \ref{lem:linf}.
\begin{proof}[Proof of Lemma \ref{lem:linf}]
Using the Cauchy-Schwarz inequality, we obtain 
\[
\B|\f{ \om }{\sin \f{x}{2}} \B|=  \f{1}{ |\sin \f{x}{2}| } |\int_0^x \om_x dx |
\les \f{1}{ |\sin \f{x}{2}| }  (\int_0^x \sin^2 \f{x}{2} dx )^{1/2} || \om ||_{\cH} 
 \les || \om ||_{\cH}
 \]
 for any $ x \in S^1$, which concludes the proof.
\end{proof}

\vspace{0.2in}
\noindent
{\bf Acknowledgments.} The author would like to thank Thomas Hou for helpful comments on an earlier version of this work. This research was supported in part by grants DMS-1907977 and DMS-1912654
from the National Science Foundation.

%Please add an acknowledgment to my two NSF grants, NSF DMS-1907977 and DMS-1912654.

\bibliographystyle{plain}
\bibliography{selfsimilar}

\end{document}